\documentclass[12pt]{article}
\UseRawInputEncoding
\usepackage{amssymb}
\usepackage{latexsym}
\usepackage{amsfonts}
\usepackage{amsmath}
\usepackage[dvips]{color}

 \numberwithin{equation}{section}

\newtheorem{thm}{Theorem}[section]
\newtheorem{proposition}[thm]{Proposition}
\newtheorem{remark}[thm]{Remark}
\newtheorem{definition}[thm]{Definition}
\newtheorem{corollary}[thm]{Corollary}

\newtheorem{lemma}[thm]{Lemma}
\newtheorem{assumption}[thm]{Assumption}

\newcommand{\qed}{\hfill\rule{2mm}{3mm}\vspace{4mm}}

\def\cB{\mbox{${\cal B}$}}

\def\cF{\mbox{${\cal F}$}}
\def\cH{\mbox{${\cal H}$}}
\def\cL{\mbox{${\cal L}$}}
\def\cG{\mbox{${\cal G}$}}

\def\cT{\mbox{${\cal T}$}}

\def\proof{\noindent {\bf Proof. }\ }

\def\wh{\widehat}

\def\qed{\hfill $\square$ \bigskip}

\def\beq{\begin{equation}}               
\def\eeq{\end{equation}}                 
\def\bea{\begin{eqnarray}}             
\def\eea{\end{eqnarray}}               
\def\be*{\begin{eqnarray*}}             
\def\ee*{\end{eqnarray*}}               
\def\ba{\begin{array}}                  
\def\ea{\end{array}}                    

\def\;{\vspace{3mm} \\ }
\def\Q{\rm Q}
\def\N{\mathbb{N} }
\def\H{\mathbb{H} }

\def\K{\bf\rm{K} }
\def\R{\mathbb{R}}
\def\E{\mathbb{E} }
\def\P{\mathbb{P} }

\def\LL{\mathbb{L} }

\def\~{\widetilde}

\def\F{{\cal F}}
\def\eps{\varepsilon}

\def\beqlb{\begin{eqnarray}} \def\eeqlb{\end{eqnarray}}
\def\beqnn{\begin{eqnarray*}} \def\eeqnn{\end{eqnarray*}}

\def\<{\langle}  \def\>{\rangle}

\def\wh{\widehat}

\def\A{{\cal A}}  \def\L{{\cal L}}

\def\bde{\begin{definition}}
\def\ede{\end{definition}}

\def\bth{\begin{thm}}
\def\eth{\end{thm}}
\def\bpr{\begin{proposition}}
\def\epr{\end{proposition}}
\def\ble{\begin{lemma}}
\def\ele{\end{lemma}}
\def\bcor{\begin{corollary}}
\def\ecor{\end{corollary}}
\def\bre{\begin{remark}}
\def\ere{\end{remark}}

\begin{document}

\def\ee{\varepsilon}
\def\qed{{\hfill $\Box$ \bigskip}}
\def\MM{{\cal M}}
\def\BB{{\cal B}}
\def\LL{{\cal L}}
\def\FF{{\cal F}}
\def\EE{{\cal E}}
\def\QQ{{\cal Q}}
\def\AA{{\cal A}}

\def\R{{\bf R}}
\def\N{{\mathbb N}}
\def\L{{\mathbb L}}
\def\E{{\bf E}}
\def\F{{\bf F}}
\def\H{{\bf H}}
\def\P{{\bf P}}
\def\Q{{\bf Q}}
\def\S{{\bf S}}
\def\J{{\bf J}}
\def\K{{\bf K}}
\def\F{{\bf F}}
\def\A{{\bf A}}
\def\loc{{\bf loc}}
\def\eps{\varepsilon}
\def\semi{{\bf semi}}
\def\wh{\widehat}
\def\pf{\noindent{\bf Proof.} }
\def\dim{{\rm dim}}

\title{\Large \bf Spine decomposition for   branching  Markov processes and its applications}
\author{ \bf Yan-Xia
Ren\footnote{Supported by NSFC
(Grant No.  11671017 and 11731009)\hspace{1mm} } \hspace{1mm}\hspace{1mm} and
Renming Song\thanks{Supported in part by a grant from the Simons Foundation (\#429343 Renming Song)} \hspace{1mm} }

\date{}
\maketitle

{\narrower{\narrower

\centerline{\bf Abstract}

\bigskip
In the literature, the spine decomposition of branching Markov processes was constructed under the assumption that each individual has at least one child.
In this paper, we give a detailed
construction of the spine decomposition of general branching Markov processes
allowing the possibility
of no offspring when a particle dies.
Then we give some applications of the spine decomposition.
\bigskip


\medskip


\par}\par}

\bigskip

\begin{doublespace}

\section{Introduction}
  \setcounter{equation}{0}

The spine decomposition theorem is a very useful tool in studying the asymptotic behaviors of
various branching models. This method was first
introduced in  \cite{LPP} for Galton-Watson processes to
give  probabilistic proofs of the Kesten-Stigum $L\log L$ theorem in supercritical case and results on the rate of decay of the survival probability in the critical and subcritical cases.
Since then, this method has been generalized to
various other  models with the branching property,
see \cite{BK, CRS, CRY, CRY2, EKW, HH1, HW, KLPP, K2, KLMR, KPR, LRS09, LRS, L, RSS, RY, YR}, for instance.

The spine decomposition theorems for Galton-Watson processes and superprocesses are pretty satisfactory. However, the spine decomposition theorem for branching Markov processes was
only proved under the assumption that each individual has at least one child, see \cite{HH1, LRS}. In \cite{P, Wang}, the spine decomposition theorem for branching Markov processes without the assumption above about the offspring distribution was used, but no detailed construction of the spine decomposition was given. The purpose of this paper is to give a detailed construction of
the spine decomposition for branching Markov processes without assuming that each individual has at least one child.
We also do not assume that the branching Markov process is supercritical, so our spine decomposition works also in the critical and subcritical case.
We also give some applications of the spine decomposition theorem.

We now introduce the setup of this paper.
We always assume that $E$ is a  locally compact separable
metric space and  $m$ is a $\sigma$-finite
Borel measure on $E$ with full support.
We will use $E_{\Delta}:=E\cup\{\Delta \}$ to denote
the one-point compactification of $E$. We will use ${\cal B}(E)$ and
${\cal B}(E_{\Delta})$ to denote the Borel $\sigma$-fields on $E$
and $E_{\Delta}$ respectively. $\cB_b(E)$ (respectively,
$\cB^+(E)$) will denote the set of all bounded
(respectively, non-negative)
$\cB(E)$-measurable functions on $E$. All functions
$f$ on $E$ will be automatically extended to $E_{\Delta}$ by setting
$f(\Delta)=0$. Let ${\bf M}_p(E)$ be the space of finite point
measures on $E$, that is, measures of the form $\mu=\sum^n_{i=1}
\delta_{x_i}$ where $n=0, 1, 2, \dots$ and $x_i\in E, i=1, \dots,
n$. (When $n=0$, $\mu$ is the trivial zero measure.) For any
function $f$ on $E$ and any measure $\mu \in {\bf M}_p(E)$, we use
$\<f, \mu\>$  or $\mu(f)$ to denote the integral of $f$ with respect to $\mu$.

We will always assume that $Y=\{Y_t, \Pi_x, \zeta\}$ is a Hunt process
on $E$ with reference measure $m$,
where $\zeta=\inf\{t>0:\,  Y_t=\Delta\}$ is the lifetime of
$Y$. Let $\{P_t, t\ge 0\}$ be the transition semigroup of $Y$:
$$
P_tf(x)=\Pi_x[f(Y_t)]\quad \mbox{ for }f\in{\cal B}^+(E).
$$
$\{P_t, t\ge 0\}$ can be extended to a strongly continuous semigroup
on $L^2(E, m)$.

Consider a branching system determined by the following three
parameters:
\begin{description}
\item{(a)} a Hunt process $Y=\{Y_t, \Pi_{x},\zeta\}$
with state space  $E$;
\item{(b)} a nonnegative bounded Borel function $\beta$ on $E$;
\item{(c)}  offspring distributions
$\{(p_n(x))_{n=0}^\infty;\, x\in E\}$, such that, for each $n\ge 0$, the function $p_n(x)$ is Borel.
\end{description}
Put \beq \psi(x,z)=\sum^\infty_{n=0}p_n(x)z^n,\quad z\geq 0. \eeq
$\psi(x, \cdot)$ is the generating function of the distribution $(p_n(x))_{n=0}^\infty$.

The branching Hunt process is characterized by the following properties:
\begin{description}
\item{(i)} Each particle has a random birth and a random death
time. \item{(ii)} Given that a particle is born at $x\in E$, the
conditional distribution of its path after birth is determined by
$\Pi_{x}$. \item{(iii)} Given the path $Y$ of a particle and given
that the particle is alive at time $t$, its probability of dying in
the interval $[t, t + \mathrm dt)$ is $\beta(Y_t)\mathrm dt + o(\mathrm dt)$.
\item{(iv)} When a particle dies at $x\in E$, it
splits into $n$ particles at $x$ with probability $p_n(x)$.
\item{(v)} The point $\Delta$ is a cemetery. When a particle reaches
$\Delta$, it stays at $\Delta$ forever and there is no branching at
$\Delta$.
\end{description}

In this paper, to avoid triviality, we always assume that $m(\{x\in E, \beta(x)>0\})>0$.
We assume that
$A(x):=\psi'(x,1)=\sum^\infty_{n=0}np_n(x)$
is bounded.

For any $c\in\cB_b(E)$, we define
$$
e_c(t)=\exp\left(-\int^t_0c(Y_s)\mathrm ds\right).
$$
Let $X_t(B)$ be the number of particles which are alive and located in $B\in\cB(E)$ at time $t$.
A particle which dies at time $t$ is not counted in $X_t(B)$ even if the death location is in $B$.
$\{X_t, t\ge 0\}$ is a Markov process in ${\bf
M}_p(E)$. This process is called a $(Y, \beta, \psi)$-branching Hunt process.
For any $\mu\in{\bf M}_p(E)$, let $\mathbf P_{\mu}$ be the law of
$\{X_t, t\ge 0\}$ when $X_0=\mu$. Then we have
 \beq
\mathbf P_{\mu}\exp\<-f, X_t\>=\exp\<\log u_t(\cdot),\mu\>,
 \eeq
where $u_t(x)$ satisfies the equation
 \beq\label{int}
u_t(x)=\Pi_{x}\left[e_{\beta}(t)\exp(-f( Y_t))+\int^t_0
e_{\beta}(s)\beta(Y_s)\psi( Y_s, u_{t-s}(
Y_s))\mathrm ds\right]\quad\mbox{for }  t\ge 0.
 \eeq

The formula \eqref{int} deals with a process started at time $0$
with one particle located at $x$, and it has a clear heuristic
meaning: the first term in the brackets corresponds to the case when
the particle is still alive at time $t$; the second term corresponds
to the case when it dies before $t$. The formula \eqref{int} implies
that
 \beq\label{int2}
u_t(x)=\Pi_{x}\int^t_0\left[\psi(Y_s,u_{t-s}(Y_s))-u_{t-s}(Y_s)\right]\beta(Y_s)\mathrm ds+
\Pi_{x}\exp(-f(Y_t))\quad\mbox{for } t\ge 0
 \eeq
(see \cite[Section 2.3]{Dy}).
For any $\mu\in{\bf M}_p(E), f\in {\cal B}_b^+ (E)$ and $t\geq 0$,
we have
\begin{eqnarray}\label{expX}
\mathbf P_{\mu} \left[\langle  f, X_t\rangle\right]
=\Pi_{\mu}\left[e_{(1-A)\beta}(t)f(Y_t)\right].
\end{eqnarray}

Let $\{P^{(1-A)\beta}_t,t\ge 0\}$ be the Feynman-Kac semigroup
defined by
$$
P^{(1-A)\beta}_tf(x):=\Pi_x\left[
e_{(1-A)\beta}(t)f(Y_t)\right], \qquad
f\in{\cal B}(E).
$$
Throughout this paper we assume that

\begin{assumption}\label{assume0}
There exist
a strictly positive Borel function $\phi$  and a constant $\lambda_1\in(-\infty,\infty)$
such that
\begin{equation}\label{invar0}
\phi(x)=e^{-\lambda_1t}P^{(1-A)\beta}_t\phi(x),\quad x\in E.
\end{equation}
\end{assumption}

Let ${\cal E}_t=\sigma(Y_s;\ s\leq t)$. Note that
$$
 \frac{\phi(Y_{t})}{\phi(x)}e^{-\lambda_1t}e_{(1-A)\beta}(t), \quad t\ge 0,
$$
is a martingale under $\Pi_x$, and so we can define a martingale
change of measure by
\begin{equation}\label{change-Pi}
\frac{\mathrm d\Pi_x^\phi}{\mathrm d\Pi_x}\Big|_{\mathcal{E}_t}= \frac{\phi(Y_{t})}
{\phi(x)}e^{-\lambda_1t}e_{(1-A)\beta}(t).
\end{equation}

For any nonzero measure $\mu\in {\bf M}_p(E)$, we define
$$
M_t(\phi):=e^{-\lambda_1 t}\frac{\langle\phi,
X_t\rangle}{\langle\phi, \mu\rangle}\qquad t\geq 0.
$$

\begin{lemma}\label{l:1.5}  For any nonzero measure $\mu\in {\bf M}_p(E)$,
$\{M_t(\phi), t\ge 0\}$ is a non-negative
martingale under $\mathbf P_{\mu}$, and
therefore there exists a limit $ M_{\infty}(\phi)\in[0,\infty)$,
$\mathbf P_{\mu}$-a.s.
\end{lemma}

\proof By the Markov property of $X$, \eqref{expX}  and
\eqref{invar0},
\begin{eqnarray*}
\mathbf P_{\mu} \left[M_{t+s}(\phi) \big| \cF_t\right] &=& \frac{1}{\<\phi,
\mu\>}e^{-\lambda_1 t} \mathbf P_{X_t} \left[e^{-\lambda_1 s} \< \phi, X_s\>\right]\\
& =&\frac{1}{\<\phi, \mu\>}e^{-\lambda_1 t} \left\<  e^{-\lambda_1
s}  \Pi_{\cdot}\left[ e_{(1-A)\beta}(s)\phi(Y_s)\right], \, X_t\right\> \\
&=&\frac{1}{\<\phi, \mu\>}e^{-\lambda_1 t} \< \phi, \, X_t\> =
M_t(\phi).
\end{eqnarray*}
This proves that $\{M_t(\phi), t\ge 0\}$ is a non-negative
$\mathbf P_\mu$-martingale and so it has an almost sure limit
$M_\infty(\phi)\in[0,\infty)$ as $t\to \infty$. \qed

It follows from the branching property that when $\mu\in {\bf
M}_p(E)$ is given by $\mu=\sum_{i=1}^n\delta_{x_i}, n=1, 2, \dots$,
$\{x_i; i=1,\cdots, n\}\subset E$, we have
$$
M_t(\phi)=\sum_{i=1}^n e^{-\lambda_1 t}\frac{\langle \phi^t,
X_t^i\rangle}{\phi(x_i)} \cdot\frac{\phi(x_i)}{\langle\phi,
\mu\rangle},
$$
where, for each $i=1, \dots, n$, $\{X^i_t, t\geq 0\}$ is a branching Hunt process starting from $\delta_{x_i}$.
If a certain assertion holds under $\mathbf P_{\delta_x}$ for all $x\in E$,
then it also holds for general $\mu$.
So in the remainder of this paper, we assume that the
initial measure is of the form $\mu=\delta_x, x\in E,$ and
$\mathbf P_{\delta_x}$ will be denoted as $\mathbf P_x$.

\section{Spine decomposition}\label{Spinedecom}

Let ${\mathbb N}=\{1, 2, \dots\}$. We will use
$$
\Gamma:= \bigcup_{n=0}^{\infty}\mathbb{N}^n
$$
(where $\mathbb{N}^0=\{\emptyset\}$) to describe the genealogical
structure of our branching Hunt process. The length (or
generation) $|u|$ of each $u\in\mathbb{N}^n$ is defined to be $n$.
When $n\ge 1$ and $u=(u_1, \dots, u_n)$, we denote $(u_1, \dots,
u_{n-1})$ by $u-1$ and call it the parent of $u$. For each $i\in
\mathbb{N}$ and $u=(u_1, \dots, u_n)$, we write $ui=(u_1, \dots,
u_n, i)$ for the $i$-th child of $u$.
More generally, for $u=(u_1, \dots, u_n), v=(v_1, \dots, v_m)\in \Gamma$,
we will use $uv$ to stand for the
concatenation $(u_1, \dots, u_n, v_1, \dots, v_m)$ of $u$ and $v$.
We will use the notation
$v<u$ to mean that $v$ is an ancestor of $u$. The set of all
ancestors of $u$ is given by $\{v\in \Gamma: \ v<u\}=\{v\in
\Gamma: \exists\ w\in \Gamma\setminus\{\emptyset\}\mbox{ such that
} vw=u\}.$
The notation $v\le u$ has the obvious meaning that either $v<u$ or $v=u$.

A subset $\tau\subset \Gamma$ is called a Galton-Watson tree if a)
$\emptyset\in\tau$; b) if $u,\ v\in\Gamma,$ then $uv\in\tau$ implies
$u\in\tau$; c) for all $u\in\tau,$
there exists $r^u\in\mathbb{N}\cup\{0\}$
such that when $j\in\mathbb{N},\ uj\in\tau$ if and only if $1\leq
j\leq r^u$. We will denote the collection of Galton-Watson trees by
$\mathbb{T}$. Each $u\in\tau$ is called a node of $\tau$ or an
individual in $\tau$ or just a particle.

To fully describe the branching Hunt process $X$, we need to
introduce the concept of marked Galton-Watson trees. We suppose that
each individual $u\in\tau$ has a mark $(Y^u,\ \sigma^u,\ r^u)$
where:
\begin{itemize}
\item[(i)] $\sigma^u$ is the lifetime of $u$,
which, along with the lifetimes of its ancestors, determines the
fission time or the death time of the particle $u$ as
$\zeta^u=\sum_{v\leq u}\sigma^{v}\ (\zeta^\emptyset =
\sigma^\emptyset)$, and the birth time of $u$ as
$b^u=\sum_{v<u}\sigma^{v}$ $(b^\emptyset =0)$;

\item[(ii)]
$Y^u: [b^u,\ \zeta^u]\rightarrow E_\Delta$ gives the location of
$u$ and $Y^u_{b^u}=Y^{u-1}_{\zeta^{u-1}}$.

\item[(iii)]$r^u $ gives the number
of the offspring of $u$ when it dies.
$r^u$ depends on $Y^u_{\zeta_u}$ in general.
\end{itemize}

We will use $(\tau,\ Y, \sigma,\ r )$ (or simply $(\tau, M)$) to
denote a marked Galton-Watson tree. We denote the set of all marked
Galton-Watson trees by $\mathcal{T}=\{(\tau,M):
\tau\in\mathbb{T}\}.$

Define
$$
\begin{array}{rl}\cF_t:=&
 \sigma\left\{\left[u, r^u,\sigma^u, (Y^u_s, s\in [b^u, \zeta^u]):
u\in\tau\in\mathbb{T}\mbox{ with }\zeta^u\le t\right]\mbox{ and
}\right.\\&\left. \left[u, (Y^u_s, s\in[b^u, t]):
u\in\tau\in\mathbb{T}\mbox{ with }t\in[b^u, \zeta^u)\right]\right\}.
\end{array}
$$ Set $\cF=\bigcup_{t\geq 0}\cF_t$.
Let $\{\mathbf P_x: x\in E\}$ be probability measures on $(\cT,\cF)$
such that
$(\cT, \cF, (\cF_t)_{t\ge 0}, (\mathbf P_x)_{x\in E})$ is the canonical
model  for $X$, the branching Hunt process in $E$. For each $x\in E$, $\mathbf P_x$ stands for the law of the branching Hunt process starting from one particle located at $x$.
For detailed constructions of the probability measures $\{\mathbf P_x: x\in E\}$,
we refer our readers to \cite{C1, C2, N}.
Under $\mathbf P_x$, the evolution of our branching Hunt process can be described as
follows.
\begin{itemize}
\item[(i)] The root moves according $\Pi_x$.
\item[(ii)] Given the path $Y^{u}_{\cdot}$ of a particle $u$ and given that $u$ is alive at time $t$, its probability of dying in the interval $[t, t+\mathrm dt)$ is $\beta(Y^u_t)\mathrm dt +o(\mathrm dt)$.
\item[(iii)] When a particle $u$ dies, it is replaced by $r^u$ number of offspring. The distribution
of $r^u$ is given by $P(Y^u_{\zeta_u})=(p_k(Y^u_{\zeta_u}))_{k\in\mathbb{N}}$.
The offspring of $u$ move independently according to  $\Pi_{Y^{u}_{\zeta_{u}}}$.
More precisely, $(Y^u_s, s\in [b^u,\ \zeta^u])$ is the restriction to $[b^u,\ \zeta^u]$
of a copy of the Hunt process starting from
 $Y^{u-1}_{\zeta^{u-1}}$ at time $b^u$.
\end{itemize}

For a marked tree $(\tau, Y, \sigma, r)$, we let  $L_t=\{u\in \tau: b^u\leq t<\zeta^u\}$ be the set of particles alive at time $t$.
Then
$$X_t=\sum_{u\in L_t}\delta_{Y^u_t}.$$

$\{M_t(\phi), t\ge 0\}$ is a $\mathbf P_x$-martingale for each $x\in E$, and so we can use
$\{M_t(\phi), t\ge 0\}$
to define a martingale change of measure of $\mathbf P_x$.
We are interested in an interpretation of the new measure,
i.e., we want to know how the process $X$ evolves under the new measure.
To this end, we need to define a new sample space $\widetilde{\mathcal{T}}$, which is  the space of marked trees with distinguished spines.
For a marked tree $(\tau, Y, \sigma, r)$, we let
$D_t=\{u\in \tau: \zeta^u\leq t, r^u=0\}$ be the set of particles that
died, before or at time $t$, with no offspring.
Let $\dag$ be a fictitious node not in $\tau$.
A spine $\xi$ on a marked tree $(\tau, Y, \sigma, r)$  is a subset of $\tau\cup \{\dag\}$ such that
\begin{itemize}
\item $\emptyset\in \xi$ and $|\xi\cap (L_t\cup \{\dag\})|=1$ for all $t\ge 0$.
\item If $v\in \xi$ and $u<v$, then $u\in \xi$.
\item If $v\in \xi$ and $r^v>0$, then there exists a unique $j=1, \dots, r^v$ with $vj\in \xi$.
If $v\in \xi$ and $r^v=0$, then $\xi\cap L_t$ is empty for all $t\ge \zeta^v$.
In this case, we will write $v=\dag -1$.
\end{itemize}
Note that the spine only contains information about the nodes along the spine, does not know  the fission times or the number of offspring at these fission times.
The fictitious particle (or node) $\dag$ might move in space, but its movement will be of
no concern to us. Thus we call $\zeta^{\dag-1}$ the ``lifetime'' of the spine. $\dag$ lives on
forever.

We write
$$
\widetilde{\mathcal{T}}=\{(\tau, Y, \sigma, r, \xi): (\tau, Y, \sigma, r)\in \mathcal{T}
\mbox{ and } \xi \mbox{ is a spine on } (\tau, Y, \sigma, r)\}
$$
for the space of marked trees with distinguished spines.

Given $(\tau, Y, \sigma, r, \xi)\in \widetilde{\mathcal{T}}$ and $t\ge0$,
we let $\xi_t:= v$ be
the unique element $ v \in \xi \cap (L_t\cup \{\dag\})$. We will use $\widetilde{Y}=(\widetilde{Y}_t)_{t\geq 0}$ to denote
the spatial path followed by the spine and $n=(n_t:  \ t\geq 0)$ to
denote the counting process of fission times along the spine. More
precisely, $\widetilde Y_t=Y^u_t$ and $n_t=|u|$, if $u\in  L_t\cap\xi.$
If $\xi_t=\dag$, we set $\widetilde Y_t = Y^\dag_t$ and write $u<\dag$ if $u\in L_s$ and $u=\xi_s$ for some $s<t$.

If $v\in\xi\cap L_t$ and $r^v>0$, then at the fission time $\zeta^v$,
exactly one of its offspring continues the spine.
Let $O_v$ be the set of offspring of $v$
except the one belonging to the spine, then for any $j=1, \dots,
r^v$ such that $vj\in O_v$, we will use $(\tau,\ M)^v_j$ to denote
the marked subtree rooted at $vj$.

Now we introduce two filtrations $\{\widetilde\cF_t\}_{t\ge0}$ and $\{\cG_t\}_{t\ge0}$
on $\widetilde\cT$ by
$$
\widetilde\cF_t:=\sigma(\cF_t, \xi_s, s\le t), \quad \cG_t:=\sigma(\widetilde Y_s: s\in [0, t]),
\quad t\ge 0.
$$
$\widetilde\cF_t$ knows everything about the marked tree up to time $t$
and the nodes on the spine up to time $t$ (and thus everything about the spine up to $t\wedge \zeta^{\dag -1}$, including which nodes make up the spine, when they were born, when they died, and their family sizes).
$\cG_t$ contains all information about the path of the spine up to time $t$.

Set $\widetilde\cF:=\bigcup_{t\geq 0}\widetilde\cF_t$, $\cG:=\sigma(\widetilde Y_s: s\ge 0)$,
$\widehat\cG:=\sigma((\widetilde Y_s, \xi_s: s\ge 0), (\zeta^u: u< \dag))$
and $\widetilde\cG:=\sigma(\cG, (\xi_s:s\geq 0),
(\zeta^u, u<\dag),(r^u: u<\dag)).$
The $\sigma$-field $\cG$ knows everything about the path of the spine,
the $\sigma$-field $\widehat\cG$ knows everything about the path of  the spine and the fission times along the spine,
and the $\sigma$-field $\widetilde\cG$ knows everything about the path of  the spine, the fission times along the spine and the number of offspring born at these fission times.

As noted by Hardy and Harris \cite{HH1},
it is convenient to consider $\{\mathbf P_x, x\in E\}$ as measures on the enlarged
space $(\widetilde \cT, \cF)$, rather than on $(\cT, \cF)$.

We need to extend the probability measures $\mathbf P_x$ to probability
measures $\widetilde {\mathbf P}_x$ on $(\widetilde\cT,
\widetilde\cF)$ so that the spine is a single genealogical
line of descent chosen from the underlying tree. We will assume
that at each fission time along the spine we make a uniform choice among the
offspring, if there is at least one offspring, to decide which line of descent continues the spine
$\xi$.
If at some fission time of the spine, there is no offspring produced, we assume the spine continues with the fictitious particle $\dag$.
Then for $u\in \tau$ we have
$$
\mbox{Prob}(u\in\xi)=\prod_{v<u}\frac{1}{r^v}.
$$
It is easy to see that
$$
\sum_{u\in L_t}\prod_{v<u}\frac{1}{r^v}+\sum_{u\in D_t}\prod_{v<u}\frac{1}{r^v}=1.
$$

We first give the following representation, which is an extension of the one given in \cite{L} for the case that $p_0=0$.

\begin{thm}\label{decom-f}
Every $f\in \widetilde{\mathcal{F}}_t$  can be
written as
 \beq\label{e-decom-f}
 f= \sum_{u\in L_t}f^u(\tau, M){\mathbf 1}_{\{\xi_t=u\}}+\sum_{u\in D_t}f^u(\tau, M){\mathbf 1}_{\{\dag-1=u\}},
 \eeq
where $f^u\in \mathcal{F}_t$.
\end{thm}

\begin{proof} Suppose $f(\tau, M, \xi)\in\widetilde {\cal F}_t$.
For every $t>0$, there is a unique $u\in L_t\cup \{\dag\}$ such that $\xi_t=u$, and if $\xi_t=\dag$, then there is one unique $u\in D_t$ such that $\dag-1=u$.
Thus we have $\sum_{u\in L_t}{\mathbf 1}_{\{\xi_t=u\}}+\sum_{t\in D_t}{\mathbf 1}_{\{\dag-1=u\}}=1$, and hence
\begin{align*}f=& \sum_{u\in L_t}f(\tau, M, \xi_t){\mathbf 1}_{\{\xi_t=u\}}+\sum_{u\in D_t}f(\tau, M, \dag-1){\mathbf 1}_{\{\dag-1=u\}}\\
 =&\sum_{u\in L_t}f(\tau, M, u){\mathbf 1}_{\{\xi_t=u\}}+\sum_{u\in D_t}f(\tau, M, u){\mathbf 1}_{\{\dag-1=u\}}.\end{align*}
 Since $f\in\widetilde {\cal F}_t$, for each fixed $u\in L_t\cup D_t$,
we have $f^u:=f(\tau, M, u)\in {\cal F}_t$. Thus  \eqref{e-decom-f} is valid.
 \qed
\end{proof}

We define the measure  $\widetilde {\mathbf P}_x$ on $\widetilde\cF_t$ by
\begin{align}
&\mathrm d\widetilde {\mathbf P}_x(\tau, M, \xi)\Big|_{\widetilde
\cF_t}\nonumber\\
=&{\mathbf 1}_{\{\xi_t\in \tau\}}\mbox{d}\Pi_x(\widetilde Y) \mbox{d}
L^{\beta(\widetilde Y)}({\bf n})\prod_{v<\xi_{t}}
p_{r^v}(\widetilde Y_{\zeta^v})
\prod_{v<\xi_{t}}\frac{1}{r^v} \prod_{j:\ vj\in
O_v}\mathrm d \mathbf
 P_{\widetilde Y_{\zeta^v}}^{t-\zeta^v}((\tau, M)^v_j)\nonumber\\
&+ {\mathbf 1}_{\{\xi_t=\dag\}}\mbox{d}\Pi_x(\widetilde Y)
\mbox{d}L^{\beta(\widetilde Y)}({\bf n})
\prod_{v<\dag-1}p_{r^v}(\widetilde Y_{\zeta^v})
\prod_{v<\dag-1}\frac{1}{r^v} \prod_{j:\ vj\in
O_v}\mathrm d \mathbf
 P_{\widetilde Y_{\zeta^v}}^{t-\zeta^v}((\tau, M)^v_j),
\label{spine representation}
\end{align}
where $\Pi_x(\widetilde Y)$ is the law of the Hunt process $\widetilde Y$
starting from $x\in E$,  $L^{\beta(\widetilde Y)}({\bf n})$ is the law of a Poisson
random measure ${\bf n}=\{\{\sigma_i: i=1, \cdots, n_t\}: t\ge 0\} $
with intensity $\beta(\widetilde Y_t)\mathrm dt$ along the path of $\widetilde Y$ which gives the fission times along the spine,
$p_{r^v}(y)=\sum_{k\ge0}p_k(y){\mathbf 1}_{(r^v=k)}$ is the
probability that the individual $v$, on the spine and located at  $y\in E$, has $r^v$
offspring, and $\mathbf{P}_{x}^{t-s}((\tau, M)^v_j))$ stands for the law of
a branching Hunt process on the marked tree $(\tau, M)^v_j$, with initial particle located at $x$  time shifted by $s$.

It follows from Theorem \ref{decom-f} that for any
bounded $f\in \widetilde\cF_t$,
\begin{align*}
\widetilde {\mathbf P}_x(f|\cF_t)
=&\widetilde {\mathbf P}_x\left(\left.
 \sum_{u\in L_t}f^u(\tau, M){\mathbf 1}_{\{\xi_t=u\}}+
 \sum_{u\in D_t}f^u(\tau, M){\mathbf 1}_{\{\dag-1=u\}}
\right|\cF_t\right)\\
=&\sum_{u\in L_t}f^u(\tau, M)\prod_{v<u}\frac{1}{ r^v}+
\sum_{u\in D_t}f^u(\tau, M)\prod_{v<u}\frac{1}{ r^v}.
\end{align*}
Thus we have for any $t\ge 0$ and bounded $f\in\widetilde \cF_t$,
 \beq\label{decom-P}
\widetilde {\mathbf P}_x(f)=\mathbf P_x\left(
\sum_{u\in L_t}f^u(\tau, M)\prod_{v<u}\frac{1}{ r^v}+\sum_{u\in D_t}f^u(\tau, M)\prod_{v<u}\frac{1}{ r^v}
\right).
 \eeq
In particular,
$$
\widetilde {\mathbf P}_x(\widetilde \cT)=\mathbf P_x\left(
\sum_{u\in L_t}\prod_{v<u}\frac{1}{ r^v}+\sum_{u\in D_t}\prod_{v<u}\frac{1}{ r^v}
\right)= \mathbf P_x(1)=1,
$$
which implies $\widetilde {\mathbf P}_x$ is a probability measure. $\widetilde
{\mathbf P}_x$ is an extension of $\mathbf P_x$ onto $(\widetilde\cT,\widetilde\cF)$
and for any bounded $f\in\widetilde\cF_t$ we have
\begin{equation}\label{many-to-one}
\int_{\widetilde\cT}f\ \mathrm d\widetilde
{\mathbf P}_x=\int_{\widetilde\cT}
\left(\sum_{u\in L_t}
f^u\prod_{v<u}\frac{1}{ r^v}+\sum_{u\in D_t}f^u
\prod_{v<u}\frac{1}{ r^v}\right)
\ \mathrm d\mathbf P_x.
\end{equation}

The decomposition \eqref{spine representation} of $\widetilde {\mathbf P}_x$
suggests the following intuitive construction of the system under
$\widetilde {\mathbf P}_x$:
\begin{itemize}
\item[(i)] the root of $\tau$ is at $x$ at time 0, and the spine process
$\widetilde Y_t$ moves according to  $\Pi_x$;

\item[(ii)] given the trajectory $\widetilde{Y_\cdot}$ of the spine,
the fission times along the spine are distributed
according to $L^{\beta(\widetilde{Y})},$ where
$L^{\beta(\widetilde{Y})}$ is the law of a Poisson random measure
with intensity $\beta(\widetilde Y_t)\mathrm dt$;

\item[(iii)] at the fission time of a node $v$ on the spine, the single
spine particle is replaced by  a random number $r^v$ of offspring
with $r^v$ being distributed according to the law
$P(\widetilde Y_{\zeta^v})=(p_k(\widetilde Y_{\zeta^v}))_{k\ge 1}$;

\item[(vi)] if $r^v>0$, the spine is chosen uniformly from the $r^v$ offspring of $v$ at
the fission time of $v$;
if $r^v=0$, the spine continues  as $\dag$.

\item[(v)] if $r^v\geq 2$,
the remaining $r^v-1$ particles $vj\in O_v$ give rise to
independent subtrees $(\tau, M)^{v}_j$, which evolve as independent
subtrees determined by the probability measure
$\mathbf{P}_{\widetilde Y_{\zeta^v}}$ shifted to the time of creation.
\end{itemize}

\begin{definition}\label{con-mart}
Suppose that $(\Omega, \cH, P)$ is a probability space,
$\{\cH_t,t\ge 0\}$ is a filtration on $(\Omega, \cH)$ and that
${\cal K}$ is a sub-$\sigma$-field of $\cH$. A real-valued process
$\{U_t, t\geq 0\}$
on $(\Omega, \cF, P)$ is called a $P(\cdot |\ {\cal
K})$-martingale with respect
to $\{\cH_t,t\ge 0\}$ if (i) it is adapted to $\{\cH_t\vee{\cal
K},t\ge 0\}$; (ii) for any $t\geq 0,\ E(|U_t|)<\infty$ and (iii) for
any $t>s$,
$$
E(U_t\big|\cH_s\vee{\cal K})=
U_s,\quad{\rm a.s.}
$$
We also say that $\{U_t, t\geq 0\}$ is a martingale
with respect to $\{\cH_t,t\ge 0\}$, given
${\cal K}$.
\end{definition}

The following result is \cite[Lemma 2.3]{LRS}.

\begin{lemma}\label{mart-prod}
Suppose that $(\Omega, \cH, P)$ is a probability space,
$\{\cH_t,t\ge 0\}$ is a filtration on $(\Omega, \cH)$ and that
${\cal K}_1, {\cal K}_2$ are two sub-$\sigma$-fields of ${\cal H}$
such that ${\cal K}_1\subset{\cal K}_2$.
Assume that $\{U^1_t, t\geq 0\}$ is a
$P(\cdot |\ {\cal K}_1)$-martingale with respect to $\{{\cal
H}_t,t\ge 0\}$,
$\{U^2_t, t\geq 0\}$ is a
$P(\cdot |{\cal K}_2)$-martingale with
respect to $\{{\cal H}_t,t\ge 0\}$. If $U^1_t\in{\cal K}_2$,
$U^2_t\in\cH_t$, and $E\left(|U_t^1U_t^2|\right)<\infty$ for any
$t\ge 0$, then the product
$\{U^1_t U^2_t, t\ge 0\}$ is a
$P(\cdot |\ {\cal K}_1)$-martingale with respect to $\{\cH_t,t\ge 0\}$.
\end{lemma}

\begin{lemma}\label{HH1-results-0}
Suppose that, given the path of $\widetilde Y$,
${\bf n}=\{\{\zeta_i: i=1, \cdots, n_t\}: t\ge 0\} $ is
a Poisson random  measure with intensity $\beta(\widetilde Y_t)\mathrm dt$
along the path of $\widetilde Y$.  Then
$$
\eta_t^{(1)}:=\prod_{i\le n_t}A(\widetilde Y_{\zeta_i})
\cdot\exp\left(-\int^t_0((A-1)\beta)
(\widetilde Y_s)\mathrm ds\right), \quad t\ge 0,
$$
is an $L^{\beta(\widetilde Y)}$-martingale with respect to the
natural filtration $\{\cL_t, t\ge 0\}$ of  ${\bf n}$.
\end{lemma}

\begin{proof}  First note  that
\beq\label{Possion}
L^{\beta(\widetilde Y)}\left[\prod_{i\le n_t}A(\widetilde Y_{\zeta_i})\right]
=\exp\left(\int^t_0((A-1)\beta)(\widetilde
Y_s)\mathrm ds\right),
\eeq
which implies that  $L^{\beta(\widetilde Y)}(\eta^{(1)}_t)=1$. It is
easy to check that $\{\eta_t^{(1)}, t\ge 0\}$ is a martingale under
$L^{\beta(\widetilde Y)}$ by using the Markov property of ${\bf n}$.
We omit the details. \qed
\end{proof}

It follows from the lemma above that we can define a measure
$L^{(A\beta)(\widetilde Y)}$ by
$$
\frac{\mathrm dL^{(A\beta)(\widetilde Y)}}{\mathrm dL^{\beta(\widetilde Y)}}
\Bigg|_{\cL_t} =
\prod_{i \le {n_t}}A(\widetilde Y_{\zeta_i})
\cdot\exp\left(-\int^t_0((A-1)\beta)(\widetilde
Y_s)\mathrm ds\right).
$$

\begin{lemma}
For any $x\in E$ and $t\ge 0$, we have
\beq\label{cond-mean1}
\displaystyle\widetilde {\mathbf P}_x\left[\left.\prod_{v<\xi_{t}}
\frac{ r^v}{A(\widetilde Y_{\zeta^v})}\right|\widehat\cG\right]=1.
\eeq
\end{lemma}

\begin{proof}
It follows from \eqref{spine representation} that, given
$\widehat\cG$, for each $v<\xi_{t}$,
$$
\widetilde {\mathbf P}_x(r^v|\widehat\cG)=A(\widetilde Y_{\zeta^v}).
$$
Since, given $\widehat\cG$,  $\{r^v, v<\xi_{n_t}\}$ are independent,
we have
$$
 \widetilde {\mathbf P}_x \left(\left.\prod_{v<\xi_{t}}\frac{r^v}
{A(\widetilde Y_{\zeta^v})}\right|\widehat\cG\right)=1.
$$
\qed
\end{proof}

\begin{lemma}\label{HH1-results}
(1) The process
$$
 \widetilde\eta_t^{(1)}:=\prod_{v<\xi_{t}}
A(\widetilde Y_{\zeta^v})
\cdot\exp\left(-\int^{t\wedge \zeta^{\dag-1}}_0((A-1)\beta)
(\widetilde Y_s)\mathrm ds\right), \quad t\ge 0,
$$
is a $\widetilde {\mathbf P}_x(\cdot |\ {\cG}\vee\sigma(\zeta^{\dag-1}))$-martingale with respect to
$\{\widetilde\cF_t,t\ge 0\}$.

(2) The process
$$
 \widetilde\eta_t^{(2)}:=\prod_{v<\xi_{t}}
\frac{r^v}{A(\widetilde Y_{\zeta^v})}
={\mathbf 1}_{\{\xi_t\in L_t\}}\prod_{v<\xi_{t}}
\frac{r^v}{A(\widetilde Y_{\zeta^v})}, \quad t\ge 0,
$$
is  a $\widetilde {\mathbf P}_x(\cdot |\widehat\cG)$-martingale with respect
to $\{\widetilde\cF_t,t\ge 0\}$, where the last equality holds because  if $\xi_t=\dag$, then $r^v=0$ for $v=\dag-1$.
\end{lemma}

\begin{proof}  (1) First note that if $\xi_t\in L_t$ then $\zeta^{\dag-1}>t$, and if $\xi_t=\dag$ then $\zeta^{\dag-1}\leq t$. For $s,t\ge 0$, by the Markov property, we have
\begin{align*}
&\widetilde
{\mathbf P}_x\left[\left.\widetilde\eta^{(1)}_{t+s}\right|\widetilde\cF_t\vee\cG\vee\sigma(\zeta^{\dag-1})\right]\\
=&\widetilde {\mathbf P}_x\left[\left.\prod_{v<\xi_{t+s}}
A(\widetilde Y_{\zeta^v})
\cdot\exp\left(-\int^{(t+s)\wedge \zeta^{\dag-1}}_0((A-1)\beta)(\widetilde
Y_r)\mathrm dr\right)\right|\widetilde\cF_t\vee\cG\vee\sigma(\zeta^{\dag-1})\right]\\
=& {\mathbf 1}_{\{\xi_{t}\in L_t\}}\prod_{v<\xi_{t}}
A(\widetilde Y_{\zeta^v})\cdot
\exp\left(-\int^{t\wedge\zeta^{\dag-1}}_0((A-1)\beta)(\widetilde Y_r)\mathrm dr\right)
\\
&\cdot
\widetilde {\mathbf P}_{x}\left[\left.\prod_{\xi_{t}\le v<\xi_{t+s}}
A(\widetilde Y_{\zeta^v})\cdot
\exp\left(-\int^{s\wedge\zeta^{\dag-1}}_0((A-1)\beta)(\widetilde
Y_{r+t})\mathrm dr\right)\right|\widetilde\cF_t\vee\cG\vee\sigma(\zeta^{\dag-1})\right]\\
&+{\mathbf 1}_{\{\xi_t=\dag\}}\prod_{v<\xi_{t}}
A(\widetilde Y_{\zeta^v})\cdot
\exp\left(-\int^{t\wedge\zeta^{\dag-1}}_0((A-1)\beta)(\widetilde Y_r)\mathrm dr\right)\\
=& {\mathbf 1}_{\{\xi_{t}\in L_t\}}\widetilde\eta_t^{(1)}\exp\left(-\int^{s\wedge
\zeta^{\dag-1}}_0((A-1)\beta)(\widetilde
Y_{r+t})\mathrm dr\right)\widetilde
{\mathbf P}^x\left[\left.\prod_{\xi_{t}\le v<\xi_{t+s}}
A(\widetilde Y_{\zeta^v}) \right|
\cG\vee\sigma(\zeta^{\dag-1})\right]\\
&+{\mathbf 1}_{\{\xi_t=\dag\}}\widetilde\eta^{(1)}_t.
\end{align*}
For fixed $t>0$, given the path of $\widetilde Y$, the collection of
fission times $\{\{\zeta^v: \xi_{{t}}\le v<\xi_{{t+s}}\}: s\ge
0\} $ is a Poisson random measure  with intensity $\beta(\widetilde
Y_{t+s})\mathrm ds$, and has law $L^{\beta(\widetilde Y_{t+\cdot})}$. It
follows from \eqref{Possion} that
$$
\widetilde{\mathbf P}_x
\left[\left.\prod_{\xi_{n_{t}}\le v<\xi_{t+s}}
A(\widetilde Y_{\zeta^v})
\right|\cG\vee\sigma(\zeta^{\dag-1})\right]=\exp\left(\int^{s\wedge \zeta^{\dag-1}}_0((A-1)\beta)(\widetilde
Y_{r+t})\mathrm dr\right).
$$
Thus
$$
\widetilde {\mathbf P}_x\left[\left.\widetilde\eta^{(1)}_{t+s}
\right|\widehat\cF_t\vee\cG\vee\sigma(\zeta^{\dag-1})\right]=\widetilde\eta_t^{(1)}.
$$

(2) For $s,t\ge 0$, by the Markov property, we have
$$
\begin{array}{rl}
\widetilde{\mathbf P}_x
\left[\left.\widetilde\eta^{(2)}_{t+s}
\right|\widetilde\cF_t\vee\widehat\cG\right]
=&\displaystyle\widetilde
{\mathbf P}_x\left[\left.\prod_{v<\xi_{{t+s}}}\frac{r^v}
{A(\widetilde Y_{\zeta^v})}
\right|\widetilde\cF_t\vee\widehat\cG\right]\\
=& \displaystyle {\mathbf 1}_{\{\xi_{t}\in L_t\}}\prod_{v<\xi_{t}}\frac{r^v}
{A(\widetilde Y_{\zeta^v})}\cdot
\displaystyle\widetilde
{\mathbf P}_{x}\left[\left.\prod_{\xi_{t}\leq
v<\xi_{{s+t}}}\frac{r^v}
{A(\widetilde Y_{\zeta^v})}\right|\widehat\cG\right]\\
=&\displaystyle \widetilde\eta^{(2)}_t,
\end{array}
$$ where in the last
equality we used \eqref{cond-mean1}. Thus we have
$$
\widetilde {\mathbf P}_x\left[\left.\widetilde\eta^{(2)}_{t+s}\right|
\widetilde\cF_t\vee\widehat\cG\right]=\widetilde\eta^{(2)}_t.
$$
\qed
\end{proof}

The effect of a change of measure using the
martingale $\{\widetilde\eta_t^{(1)}, t\ge 0\}$ will
change the fission rate along the spine from $\beta(\widetilde
Y_t)$ to $(A\beta)(\widetilde Y_t)$.
The effect of a change of measure using the martingale
$\{\widetilde\eta_t^{(2)}, t\ge 0\}$ will change the offspring distribution from
 $P(\widetilde{Y}_{\zeta_i})=(p_k(\widetilde Y_{\zeta_i}))_{k\ge1}$
to the size-biased distribution
$\dot P(\widetilde{Y}_{\zeta_i})=(\dot p_k(Y_{\zeta_i}))_{k\ge 1}$,
where $\dot p_k(y)$ is defined by
$$
\dot p_k(y)=\frac{kp_k(y)}{A(y)},\quad  k\ge 1,y\in E.
$$

Define
$$
\widetilde\eta^{(3)}_t(\phi):=\frac{\phi(\widetilde
Y_{t\wedge \zeta^{\dag-1}})}{\phi(x)}\exp\left(-\int_0^{t\wedge \zeta^{\dag-1}}
(\lambda_1-(A-1)\beta)(\widetilde Y_s)\mathrm ds
\right), \quad t\ge 0.
$$
$\{\widetilde\eta^{(3)}_t(\phi), t\ge 0\}$ is a $\widetilde {\mathbf P}_x$-martingale
with respect to $\{\cG_t\vee \sigma(\zeta^{\dag-1}), t\ge 0\}$,
and it is also
a $\widetilde {\mathbf P}_x$-martingale with respect to $\{\widetilde\cF_t,
t\ge 0\}$, since $\widetilde\eta_t^{(3)}(\phi)$ can be expressed as
\begin{align}\label{decom-eta}
\widetilde\eta^{(3)}_t(\phi)=&\sum_{u\in
L_t}\phi(x)^{-1}\phi(\widetilde
Y^u_t)\exp\left(-\int_0^{t}{(\lambda_1-(A-1)\beta)}(\widetilde
Y_s)\mathrm ds\right){\mathbf 1}_{\{\xi_t=u\}}\nonumber\\
&+\sum_{u\in
D_t}\phi(x)^{-1}\phi(\widetilde Y^u_{\zeta^u})\exp\left(-\int_0^{\zeta^u}{(\lambda_1-(A-1)\beta)}
(\widetilde Y_s)\mathrm ds{\mathbf 1}_{\{\dag-1=u\}}\right).
\end{align}
Define
$$
\eta_t(\phi):=\widetilde\eta_t^{(1)}\widetilde\eta_t^{(2)}
\widetilde\eta^{(3)}_t(\phi), \quad t\ge 0.
$$
It is easy to check, by the definition of $\widetilde\eta_t^{(1)}$, $\widetilde\eta_t^{(2)}$,
and $\widetilde\eta_t^{(3)}(\phi)$, that
\begin{equation}\label{def-eta2}
\widetilde\eta_t(\phi)={\mathbf 1}_{\{\xi_t\in L_t\}}\prod_{v<\xi_{t}}r^v\frac{\phi(\widetilde
Y_{t})}{\phi(x)}e^{-\lambda_1 t}.
\end{equation}

\begin{lemma}
$\{\widetilde \eta_t(\phi), t\ge 0\}$ is a
$\widetilde {\mathbf P}_x$-martingale with respect to
$\{\widetilde\cF_t, t\ge 0\}$.
\end{lemma}

\begin{proof}
$\{\widetilde\eta_t^{(1)}, t\ge 0\}$ is a
$\widetilde {\mathbf P}_x(\cdot |\
{\cG}\vee\sigma(\zeta^{\dag-1}))$-martingale with respect to $\{\widetilde\cF_t,t\ge 0\}$,
and $\{\widetilde\eta_t^{(2)}, t\ge 0\}$ is a
$\widetilde {\mathbf P}_x(\cdot
|\widehat\cG)$-martingale with respect to $\{\widetilde\cF_t,t\ge
0\}$. Note that $\cG\vee\sigma(\zeta^{\dag-1})\subset\widehat\cG$, and
$\widetilde\eta^{(1)}_t\in\widehat \cG$,
$\widetilde\eta^{(2)}_t\in\widetilde \cF_t$ for any $t\ge 0$. Using
Lemma \ref{mart-prod},
$\{\widetilde\eta^{(1)}_t\widetilde\eta^{(2)}_t, t\ge 0\}$ is a
$\widetilde {\mathbf P}_x(\cdot|\cG\vee\sigma(\zeta^{\dag-1}))$-martingale
with respect to $\{\widetilde\cF_t,t\ge
0\}$.
Note that $\widetilde\eta^{(3)}_t(\phi)\in \cG\vee\sigma(\zeta^{\dag-1})$ and
$\widetilde\eta^{(1)}_t\widetilde\eta^{(2)}_t\in\widetilde \cF_t$
for any $t\ge 0$. Using Lemma \ref{mart-prod} again, we see that
$\{\widetilde \eta_t(\phi)=\widetilde\eta^{(1)}_t\widetilde\eta^{(2)}_t\widetilde\eta^{(3)}_t(\phi), t\ge 0\}$
is a $\widetilde {\mathbf P}_x$-martingale with respect to
$\{\widetilde\cF_t,t\ge 0\}$.
\qed
\end{proof}

\begin{lemma}\label{proj-M} $M_t(\phi)$ is the projection of $\widetilde\eta_t(\phi)$
onto $\cF_t$, i.e.,
$$M_t(\phi)=\widetilde {\mathbf P}_x(\widetilde\eta_t(\phi)|\cF_t).$$
\end{lemma}
\begin{proof}
By \eqref{def-eta2},
$$
\widetilde\eta_t(\phi)=\sum_{u\in L_t}\prod_{v<u}r^ve^{-\lambda_1
t}\phi(x)^{-1}\phi(Y^u_t){\mathbf 1}_{\{\xi_t=u\}}.
$$
Thus
$$
\begin{array}{rl}\widetilde {\mathbf P}_x(\widetilde\eta_t(\phi)|\cF_t)
=&\displaystyle\sum_{u\in
L_t}e^{-\lambda_1t}\phi(x)^{-1}\phi(Y^u_t)\prod_{v<u}r^v\
\widetilde {\mathbf P}_x({\mathbf 1}_{\{\xi_t=u\}}|\cF_t)\\
=&\displaystyle\sum_{u\in
L_t}e^{-\lambda_1t}\phi(x)^{-1}\phi(Y^u_t)=M_t(\phi),
\end{array}
$$
where in the second equality we used the fact that
$$
\widetilde
{\mathbf P}_x({\mathbf 1}_{ L_t }(u){\mathbf 1}_{\{\xi_t=u\}}|\cF_t)={\mathbf 1}_{ L_t }(u){\mathbf 1}_{\{\xi_t=u\}}\prod_{v<u}\frac1{r^v}.
$$
\qed
\end{proof}

Now we define a probability measure $\widetilde {\mathbf Q}_x$ on
$(\widetilde\cT, \widetilde\cF)$ by
\begin{equation}\label{def-tildeQ}
\frac{\mathrm d \widetilde {\mathbf Q}_x}{\mathrm d \widetilde
{\mathbf P}_x}\left|_{\widetilde\cF_t}\right.=
\widetilde\eta_t(\phi), \quad t\ge 0,
\end{equation}
which, by \eqref{def-eta2}, says that on $\widetilde\cF_t$,
$$
\mathrm d \widetilde {\mathbf Q}_x=\widetilde
\eta_t(\phi)\mathrm d \widetilde {\mathbf P}_x={\mathbf 1}_{\{\xi_t\in L_t\}}\prod_{v<\xi_{t}}r^v\frac{\phi(\widetilde
Y_{t})}{\phi(x)}e^{-\lambda_1 t}\mathrm d \widetilde {\mathbf P}_x.
$$
Hence we have $\widetilde {\mathbf Q}_x(\xi_t\in L_t)=1$ for any $t\geq 0$, which implies that $\widetilde {\mathbf Q}_x(\xi_t\in L_t, \forall t\geq 0)=1$.
$$
\begin{array}{rl}\mathrm d \widetilde {\mathbf Q}_x
=&\displaystyle {\mathbf 1}_{\{\xi_t\in L_t\}} \frac{\phi(\widetilde
Y_{t})}{\phi(x)}\exp\left(-\int_0^{t}
(\lambda_1-(A-1)\beta)(\widetilde
Y_s)\mathrm ds\right)\mathrm d\Pi_x(\widetilde Y)\\
&\displaystyle\times \exp\left(-\int^t_0((A-1)\beta)(\widetilde
Y_s)\mathrm ds\right) \mathrm dL^{\beta(\widetilde Y)}
\prod_{v<\xi_{t}}p_{r^v}(\widetilde Y_{\zeta^v})
\prod_{j:\ vj\in O_v}\mathrm d
\mathbf P_{\widetilde Y_{\zeta^v}}^{t-\zeta^v}((\tau, M)^v_j)\\
=&\displaystyle {\mathbf 1}_{\{\xi_t\in L_t\}}\mathrm d\Pi^\phi_x(\widetilde Y)
\mathrm dL^{A\beta(\widetilde Y)}({\bf n})
\prod_{v<\xi_{t}} \frac{p_{r_v}(\widetilde Y_{\zeta^v})}{A(\widetilde Y_{\zeta^v})}
\prod_{j:\ vj\in O_v}\mathrm d\mathbf P_{\widetilde
Y_{\zeta^v}}^{t-\zeta^v}((\tau, M)^v_j)\\
=&\displaystyle {\mathbf 1}_{\{\xi_t\in L_t\}}\mathrm d\Pi^\phi_x(\widetilde Y)
\mathrm dL^{A\beta(\widetilde Y)}({\bf n})
\prod_{v<\xi_{t}}  \dot p_{r^v}(\widetilde Y_{\zeta^v})
\prod_{v<\xi_{t}}\frac{1}{r^v} \prod_{j:\ vj\in O_v}\mathrm d
\mathbf P_{\widetilde Y_{\zeta^v}}^{t-\zeta^v}((\tau, M)^v_j)\\
=&\displaystyle \mathrm d\Pi^\phi_x(\widetilde Y)
\mathrm dL^{A\beta(\widetilde Y)}({\bf n})
\prod_{v<\xi_{t}}  \dot p_{r^v}(\widetilde Y_{\zeta^v})
\prod_{v<\xi_{t}}\frac{1}{r^v} \prod_{j:\ vj\in O_v}
\mathrm d\mathbf P_{\widetilde Y_{\zeta^v}}^{t-\zeta^v}
((\tau, M)^v_j).\end{array}
$$
Thus the change of measure from $\widetilde {\mathbf P}_x$ to $\widetilde {\mathbf Q}_x$
has three effects: the spine will be changed to a Hunt process
with law $\Pi^\phi_x$,
its fission times will be changed and the
distribution of its family sizes will be sized-biased. More
precisely, under $\widetilde {\mathbf Q}_x$:
\begin{itemize}
\item[(i)] the root of $\tau$ is at $x$ at time 0, and the spine process
$\widetilde Y_t$ moves according to  $\Pi^\phi_x$;

\item[(ii)] given the trajectory $\widetilde{Y_\cdot}$ of the spine,
the fission times along the spine are distributed
according to $L^{(A\beta)(\widetilde{Y})},$ where
$L^{(A\beta)(\widetilde{Y})}$ is the law of a Poisson random measure
with intensity $(A\beta)(\widetilde Y_t)\mathrm dt$;

\item[(iii)] at the fission time of node $v$ on the spine, the single
spine particle is replaced by  a random number $r^v$ of offspring
with $r^v$ being distributed according to the law
$\dot P(\widetilde Y_{\zeta^v}):=(\dot p_k(\widetilde Y_{\zeta^v}))_{k\ge 1}$;

\item[(vi)] the spine is chosen uniformly from the $r^v$ offspring of $v$ at
the fission time of $v$;

\item[(v)]
the remaining $r^v-1$ particles $vj\in O_v$ give rise to
independent subtrees $(\tau, M)^{v}_j$, which evolve as independent
subtrees determined by the probability measure
$\mathbf{P}_{\widetilde Y_{\zeta^v}}$ shifted to the time of creation.
\end{itemize}

We define a measure $\mathbf Q_x$ on $(\widetilde\cT, \cF)$ by
$$
\mathbf Q_x:=\widetilde {\mathbf Q}_x|_{\cF}.
$$

\begin{thm}[Spine decomposition]\label{l:Qx}
$\mathbf Q_x$
is a martingale change of measure by the martingale
$\{M_t(\phi), t>0\}$: for any $t>0$,
$$
\left.\frac{\mathrm d \mathbf Q_x}{\mathrm d \mathbf P_x}\right|_{{\cal F}_t}=M_t(\phi).
$$
\end{thm}
\begin{proof} The result actually follow from a more general observation
that if $\widetilde\mu_1$ and $\widetilde\mu_2$ are two measures defined on a measure space $(\Omega, \widetilde {\cal S})$ with
Radon-Nikodym derivative
$$\frac{\mathrm d\widetilde\mu_2}{\mathrm d\widetilde\mu_1}= f,$$
and if ${\cal S}$ is a sub-¦Ò-algebra of $\widetilde {\cal S}$, then the two measures $\mu_1:= \widetilde \mu_1|_S$ and
$\mu_2:= \widetilde \mu_2|_S$ on $(\Omega, {\cal S})$ are related by the conditional expectation operation:
$$
\frac{\mathrm d\mu_2}{\mathrm d\mu_1}= \widetilde\mu_1(f|{\cal S}).
$$
For each fixed $t>0$, applying this general result with $(\Omega, \widetilde {\cal S})=(\widetilde {\cal T}, \widetilde{\cal F}_t)$, ${\cal S}={\cal F}_t$, $\widetilde\mu_2=\widetilde {\mathbf Q}_x$, and $\widetilde\mu_1=\widetilde {\mathbf P}_x$, and using Lemma \ref{proj-M}
yield the desired result.
\qed
\end{proof}

We still use $X_t(B)$ to denote the number of particles located in $B\in\cB(E)$ at
time $t$ in the marked tree with distinguished spine. Note that
 $$X_t(B)={\mathbf 1}_B(\widetilde Y_t)+\sum_{u\in L_t, u\neq\xi_{t}}{\mathbf 1}_B(Y^u_t).$$
The individuals $\{u\in L_t, u\neq\xi_{t}\}$ can be partitioned
into subtrees created from fissions along the spines,
and regarded as immigrants.
We may use the language of immigration to describe the system as follows:
under $\mathbf Q_x$,
(i) the spine process
$\widetilde Y_\cdot$ starts at $x$ at tome $0$, and moves according to
$\Pi^\phi_x$ and thus has infinite lifetime;
(ii) given the trajectory $\widetilde{Y_\cdot}$ of the spine,
the fission times along the spine are distributed
according to $L^{(A\beta)(\widetilde{Y})}$;
(iii) at the fission time of node $v$ on the spine,
$r^v-1$ particles are immigrated to the system
at $\widetilde Y_{\zeta^v}$,  the position of the spine,
with $r^v$ being distributed according to the law
$\dot P(\widetilde Y_{\zeta^v}):=(\dot p_k(\widetilde Y_{\zeta^v}))_{k\ge 1}$;
(vi) the immigrated particles give rise to the
independent subtrees, which evolve as independent
subtrees determined by the probability measure
$\mathbf P_{\widetilde Y_{\zeta_v}}$ shifted to the time of creation.
The above Theorem \ref{l:Qx} says that $\mathbf Q_x$ is the measure change
of $\mathbf P_x$ by the martingale $\{M_t(\phi), t\geq 0\}$.

\begin{thm}\label{t:spine-decom}
We have the following  decomposition for the martingale
$\{M_t(\phi), t\ge 0\}$:
\begin{equation}\label{e:spine-decom}
\widetilde {\mathbf Q}_x\left[\phi(x)M_t(\phi)\Big|\widetilde{\cG}\right] =
\phi(\widetilde Y_t)e^{-\lambda_1t}+
\sum_{u<\xi_{t}}(r^u-1)\phi(\widetilde
Y_{\zeta^u})e^{-\lambda_1\zeta^u}.
\end{equation}
\end{thm}
\begin{proof}
We first decompose the martingale $\{\phi(x)M_t(\phi), t\ge 0\}$ as
$$
\phi(x)M_t(\phi)=e^{-\lambda_1t}\phi(\widetilde Y_{t})
+e^{-\lambda_1t}\sum_{u\in L_t, u\neq\xi_{t}}\phi(Y^u_t).
$$
The individuals $\{u\in L_t, u\neq\xi_{t}\}$ can be partitioned
into subtrees created from fissions along the spines. That is, each
node $u<\xi_{t}$ in the spine $\xi$ has given birth at time
$\zeta^u$ to $r^u$ offspring among which one has been chosen as a
node of the spine while the other $r^u-1$ individuals go off
independently to
make the subtree $(\tau, M)^u_j$.  Put
$$
X^j_t=\sum_{v\in L_t, v\in(\tau, M)^u_j}\delta_{Y^v_t}(\cdot),\quad
t\ge \zeta^u.
$$
$\{X^j_t, t\ge \zeta^u\}$ is a $(Y, \beta, \psi)$-branching Hunt
process with birth time $\zeta^u$ and starting point
$\widetilde Y_{\zeta^u}$.  Then
 \begin{equation}\label{decom}\phi(x)M_t(\phi)=e^{-\lambda_1t}\phi(\widetilde
 Y_{t})+\sum_{u<\xi_{t}}
 \sum_{j:\ uj\in O_u}M_{t}^{u,j}(\phi)
  \phi(\widetilde{Y}_{\zeta^u})e^{-\lambda_1\zeta^u},
 \end{equation}
where
$$
M^{u,j}_{t}(\phi):=e^{-\lambda_1
(t-\zeta^u)}\frac{\langle\phi, X_{t-\zeta^u}^j\rangle}
{\phi(\widetilde{Y}_{\zeta^u})}.
$$
By definition \eqref{def-tildeQ},
conditional on $\widetilde\cG$,  $uj\in O_v$ evolve as independent
subtrees determined by the probability measure
$\mathbf P_{\widetilde Y_{\zeta^u}}$ shifted to $\zeta^u$,
the time of creation. Therefore, conditional on $\widetilde\cG$,
$\{M^{u,j}_{t}(\phi), t\ge 0\}$ is a $\widetilde {\mathbf Q}_x$-martingale on the subtree
$(\tau, M)^u_j$, and therefore
$$
\widetilde {\mathbf Q}_x(M^{u,j}_{t}(\phi)|\widetilde\cG)=1.
$$
Thus taking $\widetilde {\mathbf Q}_x$ conditional expectation of \eqref{decom}
gives
$$
\widetilde {\mathbf Q}_x\left[\phi(x)M_t(\phi)\Big|\widetilde{\cG}\right]
 = \phi(\widetilde Y_t)e^{-\lambda_1t}+
 \sum_{u<\xi_{t}}(r^u-1)
 \phi(\widetilde Y_{\zeta^u})e^{-\lambda_1\zeta^u},
$$
which completes the proof.\qed
\end{proof}

\begin{thm}\label{spine-time-t}
For any $u\in\Gamma$, it holds that
$$\widetilde {\mathbf Q}_x(\xi_t=u|{\cal F}_t)=
{\mathbf 1}_{\{u\in L_t\}} \frac{\phi(Y^u_t)}{\langle\phi, X_t\rangle}.
$$
\end{thm}

\begin{proof}
It suffice to show that, for any $B\in{\cal F}_t$,
$$
\int_B{\mathbf 1}_{\{\xi_t=u\}}\mathrm d\widetilde {\mathbf Q}_x=\int_B{\mathbf 1}_{\{u\in L_t\}} \frac{\phi(Y^u_t)}{\langle\phi, X_t\rangle}\mathrm d\widetilde {\mathbf Q}_x.
$$
By definition \eqref{def-tildeQ},
\begin{align*}
\int_B{\mathbf 1}_{\{\xi_t=u\}}\mathrm d\widetilde {\mathbf Q}_x=&\int_B{\mathbf 1}_{\{\xi_t=u\}}{\mathbf 1}_{\{\xi_t\in L_t\}}\prod_{v<\xi_{t}}r^v\frac{\phi(\widetilde
Y_{t})}{\phi(x)}e^{-\lambda_1 t}\mathrm d \widetilde {\mathbf P}_x\\
=&\int_B{\mathbf 1}_{\{\xi_t=u\}}{\mathbf 1}_{\{u\in L_t\}}\prod_{v<u}r^v\frac{\phi(Y^u_t)}{\phi(x)}e^{-\lambda_1 t}\mathrm d \widetilde {\mathbf P}_x.
\end{align*}
By \eqref{decom-P},
$$
\int_B{\mathbf 1}_{\{\xi_t=u\}}\mathrm d\widetilde {\mathbf Q}_x=\int_B{\mathbf 1}_{\{u\in L_t\}}\frac{\phi(Y^u_t)}{\phi(x)}e^{-\lambda_1 t}\mathrm d \mathbf P_x.
$$
It follows from Theorem \ref{l:Qx} that for any  $A\in {\cal F}_t$,
$$
\mathbf P_x(A\cap (M_t(\phi)>0))=\mathbf P_x\left(\frac{M_t(\phi)}{M_t(\phi)}, A\cap (M_t(\phi)>0)\right)=\mathbf Q_x\left(\frac{1}{M_t(\phi)},A\right).
$$
Since $\{u\in L_t\}\subset(M_t(\phi)>0)$, we have
$$
\int_B{\mathbf 1}_{\{\xi_t=u\}}\mathrm d\widetilde {\mathbf Q}_x=\int_B{\mathbf 1}_{\{u\in L_t\}}\frac{\phi(Y^u_t)}{\langle\phi, X_t\rangle}\mathrm d \mathbf Q_x.
$$
The proof is complete.
\end{proof}
\bigskip

As consequences of the result above, we have the following

\begin{corollary}\label{special-f}
If
\begin{equation*}
 f= \sum_{u\in L_t}f^u(\tau, M){\mathbf 1}_{\{\xi_t=u\}}
\end{equation*}
with $f^u\in\mathcal{F}_t$,  then
$$
\widetilde {\mathbf Q}_x(f|{\cal F}_t)=\sum_{u\in L_t}f_u\frac{\phi(Y^u_t)}{\langle\phi, X_t\rangle}\quad\mbox{on } L_t\neq \emptyset.
$$
\end{corollary}

\begin{corollary}
If $g$ is a  Borel function on $E$ then
$$
\langle g\phi, X_t\rangle=\widetilde {\mathbf Q}_x(g(\widetilde Y_t)|{\cal F}_t)\langle\phi, X_t\rangle.$$
\end{corollary}

\begin{proof}
Writing $g(\widetilde Y_t)= \sum_{u\in L_t}g(Y^u_t){\mathbf 1}_{\{\xi_t=u\}}$ and applying Corollary \ref{special-f}, we immediately get the desired conclusion.
\end{proof}

\section{Applications}

\subsection{$L\log L$ criterion for supercritical branching Hunt processes}\label{supercritical}
In this subsection, we will use the spine decomposition to prove the $L\log L$ theorem for branching Hunt processes without assuming that each individual has at least one child.

Let $\{\widehat P_t, t\ge 0\}$ be the dual semigroup of $\{P_t, t\ge 0\}$ on
$L^2(E, m)$, that is
$$
\int_Ef(x)P_tg(x)m(\mathrm dx)=\int_Eg(x)\widehat P_tf(x)m(\mathrm dx),\quad f,g\in
L^2(E,m).
$$
We will use ${\bf A}$ and $\widehat{\bf A}$ to denote the generators of
the semigroups $\{P_t\}$ and $\{\widehat P_t\}$ on $L^2(E, m)$
respectively.
In this subsection, we will assume the following

\begin{assumption}\label{assume1}
(i) There exists a family of continuous strictly positive functions
$\{p(t,\cdot,\cdot); t>0\}$ on $E\times E$ such that for any
$(t,x)\in (0,\infty)\times E$ and $f\in\cB^+(E)$, we have
$$
P_tf(x)=\int_Ep(t, x, y)f(y)m(\mathrm dy),\quad \widehat P_tf(x)=\int_Ep(t, y,
x)f(y)m(\mathrm dy).
$$
(ii) The semigroups $\{P_t\}$ and $\{\widehat P_t\}$ are
ultracontractive, that is, for any $t>0$, there exists a constant
$c_t>0$ such that
$$p(t, x, y)\le c_t\quad\mbox{ for any } (x,y)\in E\times E.$$
\end{assumption}

Let $\{\widehat P^{(1-A)\beta}_t,t\ge 0\}$ be the dual semigroup of
$\{P^{(1-A)\beta}_t,t\ge 0\}$ on $L^2(E, m)$.
Under Assumption \ref{assume1}, we can easily show that the
semigroups $\{P^{(1-A)\beta}_t\}$ and $\{\widehat
P^{(1-A)\beta}_t\}$ are strongly continuous on $L^2(E, m)$. Moreover,
there exists a family of continuous strictly positive functions
$\{p^{(1-A)\beta}(t,\cdot,\cdot); t>0\}$ on $E\times E$ such that for any
$(t,x)\in (0,\infty)\times E$ and $f\in\cB^+(E)$, we have
$$
P^{(1-A)\beta}_tf(x) =\int_E p^{(1-A)\beta}(t,x,y)f(y)m(\mathrm dy),\quad
\widehat P^{(1-A)\beta}_tf(x) =\int_E p^{(1-A)\beta}(t,y,x)f(y)m(\mathrm dy).
$$
The generators of $\{P^{(1-A)\beta}_t\}$ and $\{\widehat
P^{(1-A)\beta}_t\}$ can be formally written as $\A+(A-1)\beta$ and
$\widehat{\A}+(A-1)\beta$ respectively.

Let $\sigma(\A+(A-1)\beta)$ and $\sigma(\widehat{\A}+(A-1)\beta)$
denote the spectra of the operators $\A+(A-1)\beta$ and
$\widehat{\A}+(A-1)\beta$, respectively. It follows from Jentzch's
Theorem (Theorem V.6.6 on page 333 of \cite{Sc} ) and the strong
continuity of $\{P^{(1-A)\beta}_t\}$ and $\{\widehat
P^{(1-A)\beta}_t\}$ that the common value $\lambda_1:= \sup {\rm
Re}(\sigma(\A+(A-1)\beta))= \sup {\rm
Re}(\sigma(\widehat{\A}+(A-1)\beta))$ is an eigenvalue of
multiplicity 1 for both $\A+(A-1)\beta$ and
 $\widehat{\A}+(A-1)\beta$, and that an eigenfunction $\phi$ of
$\A+(A-1)\beta$ associated with $\lambda_1$ can be chosen to be
strictly positive a.e. on $E$ and an eigenfunction
$\widehat{\phi}$ of $\widehat{\A}+(A-1)\beta$ associated with
$\lambda_1$ can be chosen to be strictly positive a.e. on $E$.
By \cite[Proposition 2.3]{KS1} we know that $\phi$ and
$\widehat{\phi}$ are strictly positive and continuous on $E$. We
choose $\phi$ and $\widehat{\phi}$ so that
$\int_E\phi^2(x)m(\mathrm dx)=\int_{E}\phi(x)\widehat{\phi}(x)m(\mathrm dx)=1$.
Then
\begin{equation}\label{invar}
\phi(x)=e^{-\lambda_1t}P^{(1-A)\beta}_t\phi(x),\quad
\widehat\phi(x)=e^{-\lambda_1t}\widehat
 P^{(1-A)\beta}_t\widehat{\phi}(x),\quad x\in E.
\end{equation}

Therefore Assumption \ref{assume1} implies Assumption  \ref{assume0}.  We can define $\Pi^\phi_x, x\in E,$ by a martingale change of measure, see \eqref{change-Pi}.
Then $\{Y,\ \Pi_x^\phi\}$ is a conservative Markov process,
and $\phi\widehat{\phi}$ is
the unique invariant probability density for the semigroup $P^{(1-A)\beta}_t$,
that is, for any $f\in \cB^+(E)$ and $t\ge 0$,
$$
\int_E\phi(x)\widehat{\phi}(x)P^{(1-A)\beta}_tf(x)m(\mathrm dx)=\int_Ef(x)\phi(x)\widehat{\phi}(x)m(\mathrm dx).
$$
Let $p^\phi(t, x,y)$ be the transition density of $Y$ in $E$ under
$\Pi^\phi_x$. Then
$$
p^\phi(t,x,y)=\frac{e^{-\lambda_1t}}{\phi(x)}\ p^{(1-A)\beta}(t, x,
y) \phi(y).
$$

In this subsection, we also assume the following

\begin{assumption}\label{assume2}
The semigroups $\{P^{(1-A)\beta}_t\}$ and $\{\widehat
P^{(1-A)\beta}_t\}$ are  intrinsic ultracontractive, that is, for
any $t > 0$ there exists a constant $c_t$ such that
$$
p^{(1-A)\beta}(t, x, y)\le c_t\phi(x)\widehat{\phi}(y), \quad x, y
\in E.
$$
\end{assumption}

It follows from \cite[Theorem 2.8]{KS1} that
\begin{equation}\label{IU-0}
\left|\frac{e^{-\lambda_1t}p^{(1-A)\beta}(t, x,y)} {\phi(x)
\widehat\phi(y)}-1\right|\le c\,e^{-\nu t},\quad x\in E,
\end{equation}
for some positive constants $c$ and $\nu$, which is equivalent to
\begin{equation}\label{IU}
\sup_{x\in E}\left|\frac{p^\phi(t, x,y)}{\phi(y) \widehat\phi(y)}-
1\right|\le c\,e^{-\nu t}.
\end{equation}
Thus for any $f\in\cB^+_b(E)$ we have
$$
\sup_{x\in E}\left|\int_E p^\phi(t, x, y)f(y)m(\mathrm dy) - \int_E\phi(y)
\widehat\phi(y)f(y)m(\mathrm dy)\right|\le c\,e^{-\nu t}\int_E\phi(y)
\widehat\phi(y)f(y)m(\mathrm dy).
$$
Consequently we have
\begin{equation}\label{u-convergent'}
\lim_{t\to\infty}\displaystyle\frac{\int_E p^\phi(t, x, y)
f(y)m(\mathrm dy)}{\int_E\phi(y) \widehat\phi(y)f(y)m(\mathrm dy)}=1,\quad \mbox{
uniformly for } f\in\cB^+_b(E)\mbox{ and }x\in E.
\end{equation}

We also assume that

\begin{assumption}\label{assume3}$\lambda_1>0$.\end{assumption}

The above assumption says that the branching
Hunt process is supercritical.
There  are many examples of Hunt processes
satisfying Assumptions \ref{assume1} and \ref{assume2}, see \cite[Remark 1.4]{LRS}.

The purpose of this subsection is to extend the probabilistic proof
of the Kesten-Stigum $L\log L$ theorem to branching Hunt processes
without assuming that each individual has at least one child.
Let
\beq\label{def-l}l(x)=\sum_{k=2}^\infty k\phi(x)\log^+( k\phi(x))\,
p_k(x),\quad x\in E.
\eeq
The main result of this subsection can be stated as follows.

\begin{thm}\label{maintheorem} Suppose that
$\{X_t; t\ge 0\}$ is a $(Y,\beta, \psi)$-branching Hunt process and
that Assumptions \ref{assume1}, \ref{assume2} and \ref{assume3} are satisfied.
Then $M_\infty(\phi)$ is non-degenerate under $\mathbf P_\mu$ for any
nonzero measure $\mu\in {\bf M}_p(E)$ if and only if
 \beq\label{LlogL-BH}
\int_E\widehat{\phi}(x)\beta(x)l(x)m(\mathrm dx)<\infty,
 \eeq
where $l$ is defined in \eqref{def-l}.
\end{thm}

First, we give two lemmas.  The first lemma is basically
\cite[Theorem 4.3.3]{D2}.
\begin{lemma}\label{Durrett2}
 Suppose that $\P$ and $\Q$ are two probability measures on a
 measurable  space
$(\Omega, {\cal F}_{\infty})$ with filtration $({\cal F}_t)_{t\ge 0}$,
such that for some nonnegative martingale $\{Z_t, t\ge 0\}$,
$$\frac{\mathrm d\Q}{\mathrm d\P}\Big|_{{\cal F}_t}=Z_t.$$
The limit $Z_{\infty}:=\limsup_{t\to\infty}Z_t$ therefore exists and
is finite almost surely under $\P$. Furthermore, for any $F\in{\cal
F}_{\infty}$
$$\Q(F)=\int_FZ_{\infty}\mathrm d\P+\Q(F\cap\{Z_{\infty}=\infty\}),$$
and consequently,
$$\begin{array}{rl}&(a)\quad \P(Z_{\infty}=0)=1\Longleftrightarrow
\Q(Z_{\infty}=\infty)=1\\
&(b)\quad \displaystyle\int Z_{\infty}\mathrm d\P=\displaystyle\int
Z_0\mathrm d\P\Longleftrightarrow\Q(Z_{\infty}<\infty)=1.\end{array}$$
\end{lemma}

Now we are going to give a lemma which is the key to the proof of
Theorem \ref{maintheorem}. To state this lemma, we need some  more
notation.
Note that
$$
\widetilde {\mathbf Q}_x(\xi_t\neq\dag, \forall t>0)=1
$$
and thus the lifetime the spine is $\infty$.
 We can select a line of descendants
$\xi=\{\xi_0=\emptyset,\ \xi_1, \xi_2,\cdots\},$ where
$\xi_{n+1}\in\tau$ is an offspring of $\xi_n\in\tau,\ n=0,1, \cdots$, such that $\xi_t=\xi_{n_t}, t\geq 0.$
Under $\widetilde{\mathbf{Q}}_x$, given $\widetilde\cG$,
$N_t:=\{\{(\zeta^{\xi_i},\ r^{\xi_i}): i=0, 1,2, \cdots,
n_t-1\}:t\geq 0\}$ is a Poisson point process with intensity
measure $(A\beta)(\widetilde Y_t)\mathrm dt\mathrm d\dot P(\widetilde Y_t )$, where for each $y\in E$, $\dot P(y)$ is the
size-biased distribution of $P(y)$.
To simplify notation, $\zeta^{\xi_i}$ and
$r^{\xi_i}$ will be denoted as $\zeta_i$ and $r_i$, respectively.

\begin{lemma}\label{lemma1}

(1) If $\int_E\widehat{\phi}(y)\beta(y)l(y)m(\mathrm dy)<\infty$, then
$$
\sum^\infty_{i=0}{e}^{-\lambda_1\zeta_i}r_i
\phi(\widetilde{Y}_{\zeta_i})<\infty, \quad \widetilde
{\mathbf Q}_x\mbox{-a.s.}
$$

 (2) If $
\int_E\widehat{\phi}(y)\beta(y)l(y)m(\mathrm dy)=\infty$,  then
$$
\limsup_{i\rightarrow\infty}e^{-\lambda_1\zeta_i}r_i
\phi( \widetilde{Y}_{\zeta_i})
=\infty,\quad \widetilde {\mathbf Q}_x\mbox{-a.s.}
$$
\end{lemma}

The proof of the above result goes along the same line as the proof of \cite[Lemma 3.2]{LRS}. We omit the details here.

{\bf Proof of Theorem \ref{maintheorem}.}\quad The proof heavily
depends on the decomposition \eqref{e:spine-decom}.

When $\int_E\widehat{\phi}(x)\beta(x)l(x)m(\mathrm dx)<\infty$, the first
conclusion of Lemma \ref{lemma1} says that
$$
\sup_{t>0}\widetilde
{\mathbf Q}_x\left[\phi(x)M_t(\phi)\Big|\widetilde{\cG}\right] \leq
\sum_{u\in\xi}r^u
\phi(\widetilde Y_{\zeta^u})
e^{-\lambda_1\zeta^u}+\|\phi\|_{\infty}<\infty.
$$
Fatou's lemma for conditional probability implies that
$\liminf_{t\to\infty}M_t(\phi)<\infty$,
$\widetilde {\mathbf Q}_x$-a.s.
The Radon-Nikodym derivative tells us that $\{M_t(\phi)^{-1}, t\ge 0\}$ is a
nonnegative supermartingale under $\mathbf Q_x$ and therefore has a finite
limit $\mathbf Q_x$-a.s. So $\lim_{t\to\infty}M_t(\phi)=M_{\infty}<\infty$,
$\mathbf Q_x$-a.s.
  Lemma \ref{Durrett2}
implies that in this case,
$$\mathbf P_x[M_\infty(\phi)]=\lim_{t\rightarrow\infty}\mathbf P_x[M_t(\phi)]=1.$$

When $\int_E\widehat{\phi}(x)\beta(x)l(x)m(\mathrm dx)=\infty$, using
the second conclusion in Lemma \ref{lemma1}, we can get under
$\widetilde {\mathbf Q}_x,$
\[\limsup_{t\rightarrow\infty}\phi(x)M_t(\phi)
\geq\limsup_{t\rightarrow\infty}
\phi(\widetilde Y_{\zeta_{n_t}})
(r_{n_t}-1)e^{-\lambda_1\zeta_{n_t}}=\infty.\]
This yields that $M_\infty(\phi)=\infty,\ \mathbf Q_x$-a.s.  Using Lemma
\ref{Durrett2} again, we get $M_\infty(\phi)=0,\ \mathbf P_x$-a.s. The
proof is finished.\qed

\begin{thm}\label{SLL} Suppose that
$\{X_t; t\ge 0\}$ is a $(Y,\beta, \psi)$-branching Hunt process and
that Assumptions \ref{assume1}, \ref{assume2} and \ref{assume3} are satisfied.
 Suppose \eqref{LlogL-BH}  holds, then there exists $\Omega_0\subset \Omega$ with
full probability (that is,
$\mathbf P_{x}(\Omega_0)= 1$
for every $x\in E$) such that, for every $\omega\in\Omega_0$ and for every
bounded Borel function $f$ on $E$ with compact support whose set of discontinuous
points has zero $m$-measure, we have
$$\lim_{t\to\infty}e^{-\lambda_1t}\langle f, X_t\rangle=M_\infty(\phi)\int_E\widehat\phi(x)f(x)m(\mathrm dx).$$
\end{thm}

With our spine decomposition theorem
and Theorem \ref{maintheorem},
the proof of \cite{Wang} goes through. We omit the details.

\subsection{Kolmogorov type theorem for critical branching Hunt process}

In this subsection, we use our spine decomposition to give a proof of a
Kolmogorov type theorem for critical branching Hunt processes,
see Theorem \ref{thm:Kolmogorov-type-of-theorem} below.
The key to prove this result is Lemma \ref{lem:Kolmogorov-1} below, which says that
studying the limit of $\frac{t\mathbf{P}_{x}(\langle \phi, X_t\rangle>0)}{\phi(x)}$ as $t\to\infty$ is equivalent to studying the limit of
$\int_E t\mathbf{P}_{x}(\langle \phi, X_t\rangle>0)\widehat\phi(x)m(\mathrm dx)$
as $t\to\infty$. The proof of Lemma \ref{lem:Kolmogorov-1} uses our spine decomposition.

Throughout this subsection, we assume that  Assumptions \ref{assume1} and \ref{assume2} hold. Let $\lambda_1$, $\phi$ and $\widehat\phi$ be defined as in Subsection \ref{supercritical}.
Put
\begin{equation}\label{def-V}
V(x):=\psi''(x, 1)=\sum_{k=2}^\infty k(k-1)p_k(x), \quad x\in E
\end{equation}
and
\begin{equation}\label{sigma}
\sigma^2:=\int_E\beta(y)V(y)\phi^2(y)\widehat\phi(y)m(\mathrm dy).
\end{equation}

Let $\Psi$ be the operator on ${\cal  B}_E^+$ defined by
\[
(\Psi f)(x) := \psi(x,f(x)),	\quad
f\in {\cal B}^+(E), x\in E.
\]
Recall that $f$ is automatically extended to $E_{\Delta}$ by setting $f(\Delta)=0$.
For $f\in {\cal B}^+(E)$, put
$$
V_t(e^{-f})(x):=\mathbf{P}_x(\exp\langle -f, X_t\rangle),\quad t\geq 0, x\in E.
$$
Then \eqref{int2}  can be written as
\begin{equation*}
V_t(e^{-f})(x)=P_t(e^{-f}\mathbf{1}_E)(x)+\Pi_{x}(t\geq \zeta)+\int^t_0
P_{r}\left[\left(\Psi(V_{t-r}(e^{-f}))- V_{t-r}(e^{-f})\right)\beta \right](x)\mathrm ds,
\end{equation*}
where we used the fact that $\beta(\Delta)=0$.
Note that
  $$1=\Pi_x(t<\zeta)+\Pi_x(t\geq \zeta)=P_t{\mathbf 1}_E(x)+\Pi_x(t\geq \zeta).$$
Thus we have
\begin{equation*}
1-V_t(e^{-f})(x)=P_t((1-e^{-f})\mathbf{1}_E)(x)+\int^t_0
P_{r}\left[\left(-\Psi(V_{t-r}(e^{-f}))+ V_{t-r}(e^{-f})\right)\beta\right](x)\mathrm ds,
\end{equation*}
 which can be written as
\begin{align*}
&1-V_t(e^{-f})(x)=P_t((1-e^{-f})\mathbf{1}_E)(x)\nonumber\\
&+\int^t_0
P_{r}\left[\left( A V_{t-r}(e^{-f})+1-A-\Psi(V_{t-r}(e^{-f}))+ (A-1)(1-V_{t-r}(e^{-f}))\right)\beta\right](x)\mathrm ds,
\end{align*}
 which in turn is equivalent to
 \begin{align}\label{int5}
1-V_t(e^{-f})(x)=&P^{(1-A)\beta}_t((1-e^{-f})\mathbf{1}_E)(x)\nonumber\\
&+\int^t_0
P^{(1-A)\beta}_{r}[ (AV_{t-r}(e^{-f})+1-A-\Psi(V_{t-r}(e^{-f}))\beta]\mathrm ds.
 \end{align}

We first consider the asymptotic behavior of
$$
v_t(x):=\mathbf P_x(X_t(E)=0), \quad t>0, x\in E.
$$
By monotone convergence, we have
\begin{equation*}
	v_t(x) = \lim_{\theta\to \infty}V_t(e^{-\theta \mathbf 1_E} )(x),
	\quad t> 0, x\in E.
\end{equation*}
By the Markov property of $X$, we have
\begin{equation}\begin{split}\label{eq: simigroup for small vt}
	V_t v_s(x)
	&= \mathbf P_x[e^{\langle X_t,\log\lim_{\theta\to\infty } V_s(e^{-\theta \mathbf 1_E} )\rangle}]
	= \lim_{\theta \to \infty}  \mathbf P_x[e^{\langle X_t, \log V_s(e^{-\theta \mathbf 1_E} )\rangle}]\\
	&= \lim_{\theta\to\infty} V_t V_s(e^{-\theta \mathbf 1_E} )(x)
	= v_{t+s}(x),
	\quad s,t>0, x\in E.
\end{split}\end{equation}
Using \eqref{int5}, \eqref{eq: simigroup for small vt} and $f=-\log v_s$, we get
\begin{align}\label{int-v}
1-v_{t+s}(x)=&P_t^{(1-A)\beta}((1-v_s)\mathbf{1}_E)(x)\nonumber\\
&+\int^t_0
P_{r}^{(1-A)\beta}[ \left(Av_{t-r+s}+1-A-\Psi(v_{t-r+s})\right)\beta](x)\mathrm ds.
\end{align}

Define
$$
v_\infty(x)=:\lim_{t\to\infty}v_t(x)=\mathbf P_x(\exists t>0 \mbox{ such that } X_t(E)=0).
$$

Recall the quantities $V$ and $\sigma^2$ defined in \eqref{def-V} and \eqref{sigma}.
Throughout  this subsection we assume that
\begin{assumption}\label{assume4}
(i) The branching Hunt process $X$ is critical, i.e., $\lambda_1 = 0$;

(ii)
$\sigma^2>0$;

(iii) the function $\phi V: x\rightarrow \phi(x)V(x)$  is bounded on $E$.
\end{assumption}

\begin{lemma}\label{extinction}   Suppose that Assumptions \ref{assume1},  \ref{assume2} and \ref{assume4} (i-ii) hold.
Then for any $x\in E$,
\begin{equation}\label{uniform-to-0}
\lim_{t\to\infty}\sup_{x\in E}\frac{\mathbf P_x(X_t(E)>0)}{\phi(x)}=0.
\end{equation}
\end{lemma}

\begin{proof}
For any $f,g\in {\cal B}^+(E)$, we use $\langle f,g\rangle_m$ to denote  $\int_Ef(x)g(x)m(\mathrm dx).$
Under Assumption \ref{assume2},  $\langle 1, \widehat\phi\rangle_m<\infty$. In fact, according to \eqref{IU-0}, for $t>0$ large enough, there is a $c'_t>0$ such that
$$\widehat\phi(y) \leq q(t,x,y)  (c'_t)^{-1}\phi^{-1}(x),$$
and clearly, as a function of $y$, the right hand above  is integrable with respect to $m$.
Integrating \eqref{int5} with respect to  $\widehat\phi(x)m(\mathrm dx)$,
we get that
 \beq\label{1-v-eq}
\langle 1-v_{t+s},\widehat\phi\rangle_m=\langle 1-v_s,\widehat\phi\rangle_m+\int^t_0
\langle\left( Av_{t-r+s}+1-A-\Psi(v_{t-r+s})\right)\beta,\widehat\phi\rangle_m\mathrm ds.
 \eeq
 Letting $s\to\infty$, we get
\begin{equation*}
\langle 1- v_{\infty},\widehat\phi\rangle_m=\langle 1- v_\infty,\widehat\phi\rangle_m+t
\langle\left(A v_\infty+1-A-\Psi(v_{\infty})\right)\beta,\widehat\phi\rangle_m.
\end{equation*}
Thus we have
\begin{equation*}
\langle Av_\infty+1-A-\Psi(v_{\infty}),\beta\widehat\phi\rangle_m=0.
\end{equation*}
It is easy to check that for any $x\in E$, $Az+1-A-\psi(x,z)\le0, \forall z\in[0,1]$. Since $\widehat\phi(x)>0$ on $E$, we must have
\begin{equation}\label{0-ae}
Av_\infty+1-A-\Psi(v_{\infty})=0,\quad  m\mbox{-a.e.} \mbox{ on }\{x\in E, \beta(x)>0\}.
\end{equation}
Letting $s\to\infty$ in \eqref{int-v}, we get
\begin{equation*}
1-v_\infty(x)=P_t^{(1-A)\beta}((1-v_\infty)\mathbf{1}_E)(x)+\int^t_0
P_{s}^{(1-A)\beta}[(Av_\infty+1-A-\psi(v_\infty))\beta](x)\mathrm ds,
\end{equation*}
 and thus $1-v_\infty(x)=P_t^{(1-A)\beta}((1-v_\infty)\mathbf{1}_E)(x)$, which says that $1-v_\infty$ is an eigenfunction of $\A+(A-1)\beta$ corresponding to the eigenvalue $\lambda_1=0$.
 Since the eigenvalue $\lambda_1=0$ is simple,
 $1-v_\infty=c\phi$ on $E$ for some constant $c$.
 Note that, for each fixed $x\in E$, the function $\psi_0(x,z):= \psi(x,z)-A(x)z+A(x)-1$ is strictly decreasing for $z\in(0, 1)$
with $\psi_0(x,1)=0$ and $\psi_0(x,0)=\sum^\infty_{k=2}(k-1)p_k(x)\geq 0$.
 Assumption \ref{assume4} (ii)  implies that
 $m(\{x\in E; \beta(x)>0, \psi_0(x,0)>0\})>0$.
 Since $v_\infty$ satisfies \eqref{0-ae}, we must have $c=0$, or equivalently  $v_\infty\equiv 1$. Thus $$\lim_{t\to\infty}\mathbf{P}_x(X_t(E)>0)=1-v_\infty(x)=0,\quad x\in E.$$

By \eqref{int-v} and \eqref{IU-0}, we have
$$1-v_{t+s}(x)\leq P_t^{(1-A)\beta}((1-v_s)\mathbf{1}_E)(x)\le (1+ce^{-\nu t})\phi(x)\int_E\widehat\phi(y)(1-v_s)(y)m(\mathrm dy),
$$
which implies
\begin{equation*}
\frac{1-v_{t+s}(x)}{\phi(x)}\leq (1+ce^{-\nu t})\int_E\widehat\phi(y)(1-v_s)(y)m(\mathrm dy).
\end{equation*}
Using the monotonicity of $v_t$ in $t$,
we get \eqref{uniform-to-0}.
 \qed
\end{proof}

The following Kolmogorov type theorem is the main result of this subsection.

\begin{thm}\label{thm:Kolmogorov-type-of-theorem} Suppose that Assumptions \ref{assume1}, \ref{assume2} and  \ref{assume4}  hold. Then
\begin{equation}\label{servive-rate}
\lim_{t\to\infty}\frac{t\mathbf{P}_{x}(\langle \phi, X_t\rangle>0)}{\phi(x)}=\frac{2}{\sigma^2}
\end{equation}
uniformly for $x\in E$.
\end{thm}

We prove the above  result by  proving two lemmas first.
Define
\begin{equation}\label{def-b(t)}
b(t):=\int_E(1-v_t(x))\widehat\phi(x)m(\mathrm dx)=\int_E\mathbf P_x(X_t(E)>0)\widehat\phi(x)m(\mathrm dx).
\end{equation}

\begin{lemma}\label{lem:Kolmogorov-1}
	Under Assumptions \ref{assume1}, \ref{assume2} and \ref{assume4}, we have
\[
	\sup_{x\in E}\Big | \frac{1-v_t(x)}{b(t) \phi(x)} - 1\Big |
	\xrightarrow[t\to\infty]{} 0,
\]
where $b(t)$ is defined in \eqref{def-b(t)}.
\end{lemma}

\begin{proof}
First note that
\[
\frac{1-v_t(x)}{\phi(x)}=\frac{\mathbf{P}_x(X_t(E)>0)}{\phi(x)}=\mathbf{Q}_x\left(\frac{{\mathbf 1}_{(X_t(E)>0)}}{\langle\phi, X_t\rangle}\right)=\mathbf{Q}_x\left(X_t(\phi)^{-1}\right).
\]
For $0<t_0<t<\infty$, define
\[
 I_t^{(0, t_0]}=\sum_{u\in L_{t_0}, u\neq\xi_{t_0}}\delta_{Y^u_t} \quad\mbox{ and }\quad
I_t^{(t_0,t]}=\sum_{u\in L_t\setminus L_{t_0}}\delta_{Y^u_t}.
\]
Then we have
\begin{equation}\label{eq:decomposition-on-I}
X_t=I^{(0, t_0]}_t+I^{(t_0, t]}_t.
\end{equation}
Define
$$
\mathbf{Q}_{\phi\widehat\phi m}(\cdot):=\int_E\mathbf{Q}_x(\cdot)\phi(x)\widehat\phi(x) m(\mathrm dx).
$$
Under $\mathbf{Q}_{\phi\widehat\phi m}$,
$X_0=\delta_Z$ with $Z$ being an $E$-valued random variable with distribution $\phi\widehat\phi m$.
	It is easy to see, from the construction of $\mathbf Q_x$  and the Markov property of the immigration that for any $0 < t_0 < t < \infty$,
\[
	\mathbf{Q}_x [(I_t^{(t_0,t]}(\phi))^{-1}|{\cal G}_{t_0}]
	= \mathbf{Q}_{\widetilde Y_{t_0}}[(X_{t-t_0}(\phi))^{-1}]
	= (\phi^{-1}(1-v_{t-t_0}))(\widetilde Y_{t_0}).
\]
Therefore, we have
\[\begin{split}
	\mathbf{Q}_{\phi\widehat\phi m}[(I_t^{(t_0,t]}(\phi))^{-1}]
	= \mathbf{Q}_{\phi\widehat\phi m}[(\phi^{-1}(1-v_{t-t_0}))(\widetilde Y_{t_0}) ]
	= \langle 1- v_{t-t_0},\widehat\phi\rangle_m
\end{split}\]
and
\begin{equation}
\label{eq:Yt0t}\begin{split}
	\mathbf{Q}_x[(I_t^{(t_0,t]}(\phi))^{-1}]
	= &\mathbf{Q}_x[(\phi^{-1}(1-v_{t-t_0)})(\widetilde Y_{t_0}) ]\\
	= & \int_E  p^\phi(t_0,x,y)(\phi^{-1}(1-v_{t-t_0}))(y) m(\mathrm dy).
\end{split}
\end{equation}
	By the decomposition \eqref{eq:decomposition-on-I}, we have
\begin{equation}\label{eq:vt-equation}\begin{split}
	\phi^{-1}(1-v_t(x))
	&= \mathbf{Q}_x [(X_t(\phi))^{-1}]\\
	&= \mathbf{Q}_{\phi\widehat\phi m} [(I^{(t_0,t]}_t(\phi))^{-1}] + \big( \mathbf{Q}_x[(I^{(t_0,t]}_t(\phi))^{-1}] - \mathbf{Q}_{\phi\widehat\phi m} [(I^{(t_0,t]}_t(\phi))^{-1}] \big) \\
	&\quad + \big(\mathbf{Q}_x[(X_t(\phi))^{-1} - (I^{(t_0,t]}_t(\phi))^{-1}] \big)\\
	&=: \langle 1-v_{t-t_0},\widehat\phi \rangle_m + \epsilon_x^1(t_0,t) +\epsilon_x^2(t_0,t).
\end{split}\end{equation}
Suppose that $t_0 >1$, and let $c,\nu > 0$ be the constants in \eqref{IU}. Using \eqref{eq:Yt0t}, we have
\begin{equation}\label{eq:epsilon-1}\begin{split}
	|\epsilon_x^1(t_0,t)|
	& = \big| \mathbf{Q}_x[(I^{(t_0,t]}_t(\phi))^{-1}] - \mathbf{Q}_{\phi\widehat\phi m} [(I^{(t_0,t]}_t(\phi))^{-1}] \big| \\
	& = \big|  \int_E  p^\phi(t_0,x,y)(\phi^{-1}(1-v_{t-t_0}))(y) m(\mathrm dy) - \langle 1-v_{t-t_0},\widetilde \phi\rangle_m \big|\\
	& \leq \int_{y\in E} \big| p^\phi(t_0,x,y) - (\phi\widehat\phi)(y) \big| (\phi^{-1}(1-v_{t-t_0}))(y) m(\mathrm dy)\\
	& \leq ce^{-\nu t_0}\langle 1- v_{t-t_0},\widehat\phi \rangle_m.
\end{split}\end{equation}
We also have
\begin{equation}\label{eq:epsilon-2}\begin{split}
	|\epsilon_x^2(t_0,t)|
	&= \big| \mathbf{Q}_x[(X_t(\phi))^{-1} - (I^{(t_0,t]}_t(\phi))^{-1}] \big| \\
	&= \mathbf{Q}_x[I_t^{(0,t_0]}(\phi)\cdot (X_t(\phi))^{-1}\cdot (I^{(t_0,t]}_t(\phi))^{-1}]\\
	&\leq \mathbf{Q}_x[\mathbf 1_{I_t^{(0,t_0]}(\phi)>0}\cdot (I^{(t_0,t]}_t(\phi))^{-1}]\\
	&= \mathbf{Q}_x \left(\mathbf{Q}_x[\mathbf 1_{I_t^{(0,t_0]}(\phi)>0}|{\cal G}_{t_0}] \cdot \mathbf{Q}_x[ (I^{(t_0,t]}_t(\phi))^{-1}|{\cal G}_{t_0}] \right).
\end{split}\end{equation}
Recall that $\zeta_i$ and $r_i$ are the shorthand notation for $\zeta_{\xi_i}$ and
$r_{\xi_i}$ respectively.
Note that
\begin{align*}
\mathbf{Q}_x[\mathbf 1_{I_t^{(0,t_0]}(\phi)=0}|{\cal G}_{t_0}]=&\mathbf{Q}_x\left[\prod_{\zeta_i\leq t_0}
(\mathbf P_{\widetilde Y(\zeta_i)}(X_{t-\zeta_i}(E)=0))^{r_i-1}
|{\cal G}_{t_0}\right]\\
\geq &\mathbf{Q}_x\left[\prod_{\zeta_i\leq t_0}
(\mathbf P_{\widetilde Y(\zeta_i)}(X_{t-t_0}(E)=0))^{r_i-1}
|{\cal G}_{t_0}\right]
\end{align*}
and  that
\begin{equation}\label{eq:epsilon-2-1}\begin{split}
\mathbf{Q}_x[\mathbf 1_{I_t^{(0,t_0]}(\phi)>0}|{\cal G}_{t_0}]\le &\mathbf{Q}_x\left[\prod_{\zeta_i\leq t_0}(r_i-1)
 \mathbf P_{\widetilde Y(\zeta_i)}(X_{t-t_0}(E)>0)
|{\cal G}_{t_0}\right]\\
\leq& \mathbf{Q}_x\left[\sum_{\zeta_i\leq t_0}(r_i-1)(1-v_{t-t_0})(\widetilde Y(\zeta_i))\right]\\
=&\int^{t_0}_0\beta(\widetilde Y_s)(k-1)kp_k(\widetilde Y_s)(1- v_{t-t_0})(\widetilde Y_s)\mathrm ds\\
\leq &t_0\|\beta V\phi\|_\infty\|\phi^{-1}(1-v_{t-t_0})\|_\infty.
\end{split}\end{equation}
Thus by \eqref{eq:epsilon-2} and \eqref{eq:epsilon-2-1}, we have
\begin{equation}\label{eq:epsilon-2-final}\begin{split}
 |\epsilon_x^2(t_0,t)|\leq &t_0\|\beta V\phi\|_\infty\|\phi^{-1}(1-v_{t-t_0})\|_\infty \mathbf{Q}_x\left(\mathbf{Q}_x[ (I^{(t_0,t]}_t(\phi))^{-1}|{\cal G}_{t_0}] \big]\right)\\
 =&t_0\|\beta V\phi\|_\infty\|\phi^{-1}(1- v_{t-t_0})\|_\infty\int_E  p^\phi(t_0,x,y)(\phi^{-1}(1-v_{t-t_0}))(y) m(\mathrm dy)\\
 \leq &t_0\|\beta V\phi\|_\infty\|\phi^{-1}(1- v_{t-t_0})\|_\infty(1+ce^{-\nu t})\langle 1-v_{t-t_0},\widehat\phi\rangle_m.
\end{split}\end{equation}
Combining \eqref{eq:vt-equation}, \eqref{eq:epsilon-1} and \eqref{eq:epsilon-2-final}, we have that
\begin{equation}\label{vts-inequality}\begin{split}
	\Big|\frac{\phi^{-1}(1-v_t(x))}{\langle 1-v_{t-t_0},\widehat \phi \rangle_m}-1 \Big|
	&\leq \frac{|\epsilon_x^1(t_0,t)|}{\langle 1-v_{t-t_0},\widehat\phi \rangle_m} + \frac{|\epsilon_x^2(t_0,t)|}{\langle 1- v_{t-t_0},\widehat \phi \rangle_m}\\
	&\leq ce^{-\gamma t_0} +t_0\| \beta V\phi\|_\infty \| \phi^{-1}(1- v_{t-t_0})\|_\infty (1+ce^{-\nu t_0}).
\end{split}\end{equation}
Since we know from
	Lemma \ref{extinction}
	that $\| \phi^{-1}(1-v_t)\|_\infty\to 0$ as $t\to\infty$, there exists a map $t\mapsto t_0(t)$ such that,
\[
	t_0(t)
	\xrightarrow[t\to\infty]{} \infty;
	\quad t_0(t)\| \phi^{-1}(1-v_{t-t_0(t)})\|_\infty
	\xrightarrow[t\to\infty]{} 0.
\]
Plugging this choice of $t_0(t)$ back into \eqref{vts-inequality}, we have that
\begin{equation}\label{eq:k1}
	\sup_{x\in E}\Big|\frac{\phi^{-1}(1-v_t(x))}{\langle 1-v_{t-t_0(t)},\widehat\phi \rangle_m}-1 \Big|
	\xrightarrow[t\to\infty]{} 0.
\end{equation}
	Now notice that
\begin{equation}\label{eq:k2}\begin{split}
	\Big |\frac {\langle 1-v_t, \widehat\phi\rangle_m} {\langle 1-v_{t-t_0(t)} , \widehat\phi\rangle_m} - 1 \Big |
	&\leq \int \Big | \frac{\phi^{-1}(1-v_t(x))}{\langle 1-v_{t-t_0(t)} , \widehat\phi\rangle} - 1 \Big| \phi \widehat\phi(x) m(\mathrm dx)\\
	&\leq \sup_{x\in E}\Big|\frac{\phi^{-1}(1-v_t(x))}{\langle 1-v_{t-t_0(t)},\widehat\phi \rangle_m}-1 \Big|
	\xrightarrow[t\to\infty]{} 0.
\end{split}\end{equation}
	Finally, by \eqref{eq:k1}, \eqref{eq:k2} and property of uniform convergence, we have
\[
	\sup_{x\in E}\Big|\frac{\phi^{-1}(1-v_t(x))}{\langle 1-v_{t},\widehat\phi \rangle_m}-1 \Big|
	\xrightarrow[t\to\infty]{} 0
\]
	as desired.
\qed
\end{proof}

\begin{lemma}\label{lem:Kolmogorov-2}
	Under Assumptions  \ref{assume1}, \ref{assume2} and \ref{assume4},  we have
\[
\frac{1}{tb(t)}	\xrightarrow[t\to\infty]{} \frac{1}{2}
 \langle  \beta V\phi,\phi\widehat\phi\rangle_m,
\]
where $b(t)=\langle 1-v_t,\widehat\phi\rangle_m$.
\end{lemma}
\begin{proof}
For $z\in[0,1]$, define
$$\psi_0(x,z):=\psi(x,z)-1-A(x)(z-1))
$$
and
 $$R(x,z):=\psi_0(x,z)-\frac{1}{2}V(x)(z-1)^2.$$
Note that $\psi'(x, \theta)\le 0$ and $\psi''(x, \theta)\ge 0$ for $\theta\in (z,1]$. By the mean value theorem,
$\psi_0(x,z)=\frac{1}{2}\psi''(x, \theta)(z-1)^2\le \frac{1}{2}V(x)(z-1)^2$
with $\theta\in (z,1]$.
 Thus
$$R(x,z)= e(x,z)(z-1)^2, \quad\forall z\in[0,1],$$
where $e(x,z)$ satisfies $|e(x,z)|\le V(x),\forall z\in[0,1]$ and that
\begin{equation}\label{equation:reason2}
	e(x,z)
	\xrightarrow[z\to 1]{} 0,
	\quad x\in E.
\end{equation}
Let $\Psi_0$ be the operator on ${\cal B}^+(E)$ defined by
\[
	(\Psi_0 f)(x)
	:= \psi_0(x,f(x)),
	\quad f\in {\cal B}^+(E), x\in E.
\]
 Writing $l_t(x):=(1-v_t(x))-b(t)\phi(x)$, Lemma \ref{lem:Kolmogorov-1} says that
\begin{equation}\label{equation:reason1}
	\sup_{x\in E}\Big|\frac{l_t(x)}{b(t)\phi(x)}\Big|
	\xrightarrow[t\to\infty]{} 0.
\end{equation}
Using \eqref{1-v-eq},
	we see that $t\mapsto b(t)$ is differentiable on the set
\[
\mathbf C
	=\{t> s_0: \text{the function}~ t \mapsto \langle\Psi_0(v_t),\beta\widehat\phi\rangle_m~ \text{is continuous at}~ t \}
\]
	and that
\begin{equation}\label{equation:b(t)}\begin{split}
	\frac{\mathrm d}{\mathrm dt}b(t)
	&= -\langle\Psi_0(v_t),\widehat\phi\rangle_m
	= -\big\langle \frac{1}{2} V \cdot (1-v_t)^2+R (\cdot,v_t(\cdot)),\beta\widehat\phi\big\rangle_m \\
	&= -\big\langle \frac{1}{2}  V \cdot \big(b(t)\phi+l_t\big)^2+R (\cdot,v_t(\cdot)),\beta\widehat\phi\big\rangle_m\\
	&= -b(t)^2\big[\frac{1}{2} \langle  \beta V\phi,\phi \widehat\phi\rangle_m+g(t)\big],
	\quad t\in \mathbf C,
\end{split}\end{equation}	where
\[\begin{split}
	g(t)
	&= \Big\langle \frac{l_t}{b(t) \phi}, \beta V\phi^2\widehat\phi\Big\rangle_m + \frac{1}{2}\Big\langle \Big(\frac{l_t}{b(t) \phi}\Big)^2, \beta V\phi^2\widehat\phi\Big\rangle_m + \Big\langle \frac{R(\cdot,v_t(\cdot))}{b(t)^2 \phi^2},\phi^2\widehat\phi\Big\rangle_m
\end{split}\]
It follows from \eqref{equation:b(t)} that
\begin{equation*}
	\frac{\mathrm d}{\mathrm dt} \Big(\frac{1}{b(t)}\Big)
	= -\frac{\mathrm d b(t)}{b(t)^2\mathrm dt}
	= \frac{1}{2}\langle \beta V\phi,\phi\widehat\phi\rangle_m + g(t),
	\quad t\in \mathbf C.
\end{equation*}
Since the function $t \mapsto \langle\Psi_0(v_t),\beta\widehat\phi\rangle_m$
is non-increasing,
 $(s_0, \infty)\setminus \mathbf C$ has at most countably many points.
Using \eqref{equation:reason2} and  \eqref{equation:reason1}, and repeating the argument in the proof of
 \cite[Lemma 5.4]{RSS}, we obtain that $g(t)\to 0$ as $t\to\infty$, and thus
\[
	\frac{1}{b(t)t}
	\xrightarrow[t\to\infty]{} \frac{1}{2}\langle  \beta V\phi,\phi\widehat\phi\rangle_m
\]
	as desired.
\qed
\end{proof}

Combining  Lemmas \ref{lem:Kolmogorov-1} and \ref{lem:Kolmogorov-2}, we immediately
get Theorem \ref{thm:Kolmogorov-type-of-theorem}.

\subsection{ Branching Brownian motion and traveling wave solution}
We consider a branching Brownian motion on $\mathbb{R}$, i.e.,
 the spatial motion $Y=\{Y_t, \Pi_x\}$  is a Brownian motion on $\mathbb{R}$. Suppose
 the branching rate $\beta>0$ is a
 constant, the offspring distribution $\{(p_n)_{n=0}^\infty\}$ does not depend on the spatial position  and $A:=\sum^\infty_{n=0}np_n<\infty$.

It is known that, for any $\lambda\in\mathbb{R}$,
$\Pi_xe^{\lambda Y_t}=e^{\frac{1}{2}\lambda^2 t}$. Put $\phi(x)=e^{-\lambda x}$. Then
$$\phi(x)=e^{-[\frac{1}{2}\lambda^2+(A-1)\beta] t}P^{(1-A)\beta}_t\phi(x),\quad x\in \mathbb{R}.$$
Therefore,
\begin{equation}
 W_{t}(\lambda):=e^{-[\frac{1}{2}\lambda^2+(A-1)\beta] t}\langle \phi, X_t\rangle=e^{-[\frac{1}{2}\lambda^2+(A-1)\beta] t}\sum_{u\in L_t}
  e^{-\lambda Y_u(t)}, \quad t\ge 0,
\label{martingale}
\end{equation}
is a non-negative $\mathbf P_x$-martingale with respect to  $\{{\cal F}_t, t\geq 0\}$.
Thus, for any $x\in\mathbb{R}$,
the limit $W_\infty(\lambda):=\lim_{t\to\infty}W_t(\lambda)$ exists $\mathbf P_x$-almost surely.

Under the assumption $p_0=0$,
Kyprianou \cite{K2} used spine decomposition  techniques to give necessary
and sufficient conditions for the $L^1$-convergence of the  martingales
$\{W_{t}(\lambda), t\ge 0\}$:

\begin{thm}\label{L1} Suppose $p_0=0$. Let $\underline{\lambda}:=\sqrt{2\beta(A-1)}$.

(1) if $|\lambda|\geq\underline{\lambda}$, $W_\infty(\lambda)=0$  $\mathbf P_x$-almost surely;

(2) if $|\lambda|<\underline{\lambda}$ and $\sum^\infty_{n=1}p_n n\log n=\infty$, then $W_\infty(\lambda)=0$  $\mathbf P_x$-almost surely;

(3) if $|\lambda|<\underline{\lambda}$ and $\sum^\infty_{n=1}p_n n\log n<\infty$, then $W_t(\lambda)\to W_\infty(\lambda)$ $\mathbf P_x$-almost surely and in $L^1(\mathbf P_x)$.
\end{thm}

Using spine techniques, Hardy and Harris \cite{HH1}   proved that in many cases where the martingale has a non-trivial limit, the
convergence can be strengthen as  $L^p(\mathbf P_x)$-convergence with some $p\in(1,2]$.

\begin{thm}\label{Lp} Suppose $p_0=0$.
For any $x\in\mathbb{R}$, and for each $p\in(1, 2]$ we have

(1) As $t\to\infty$, $W_t(\lambda)\to W_\infty(\lambda)$ $\mathbf P_x$-almost surely and in $L^p(\mathbf P_x)$ if $p\lambda^2<2(A-1)\beta$ and $\sum^\infty_{n=1}p_n n^p<\infty$;

(2)  $\lim_{t\to\infty}\mathbf P_x(W_t(\lambda))=\infty$ if $p\lambda^2>2(A-1)\beta$  or $\sum^\infty_{n=1}p_n n^p=\infty$.
\end{thm}
Now  using our spine decomposition in Section \ref{Spinedecom}, the above Theorems \ref{L1} and \ref{Lp} also hold  for the case that $p_0>0$.

It is known that
\begin{equation*}
\partial W_{t}(\lambda):=e^{-[\frac{1}{2}\lambda^2+(A-1)\beta] t}\sum_{u\in L_{t}}(Y_{u}(t)+\lambda t)
\mathrm{e}^{-\lambda Y_{u}(t)}
\mathrm{e}^{-\lambda Y_{u}(t)}, \quad t\ge 0,
\end{equation*}
is a $\mathbf P_x$-martingale with respect to
  $\{{\cal F}_t, t\geq 0\}$, which is also referred
to as the \textit{derivative martingale}.

The martingale $\{\partial W_{t}(\lambda), t\ge 0\}$ is  not non-negative.
To establish the convergence of $\partial W_{t}(\lambda)$ as $t\to\infty$, we usually consider the following related non-negative martingale.

Let $\widetilde{L}_{t}$ denote the set of particles in
 $L_{t}$ that, along with their ancestors, have not met by time $t$ the space-time
barrier $y+\sqrt{(A-1)\beta} t=-x$. Define
\begin{equation}
V_{t}^{x}(\lambda)=e^{-[\frac{1}{2}\lambda^2+(A-1)\beta] t}\sum_{u\in
\widetilde{L}_{t}}\frac{x+Y_{u}(t)+\lambda t}{x}
\mathrm{e}^{-\lambda Y_{u}(t)}, \quad t\ge 0.
\label{v}.
\end{equation}
Under the condition that $p_0=0$, Kyprianou \cite{K2} proved that
$\{V_{t}^{x}(\lambda), t\ge 0\}$ is a mean 1
$\mathbf P_x$-martingale, and that when $\lambda\ge \sqrt{2(A-1)\beta} $, $\partial
W(\lambda):=\lim _{t\to +\infty}\partial W_{t}(\lambda)$ $\mathbf P_x$-a.s. exists and
is equal to $\lim_{t\to +\infty}xV_{t}^{x}(\lambda).$

The importance of the limit $\partial W(\underline{\lambda})$ lies in that when $\partial W(\underline{\lambda})$ is non-degenerate, its rescaled Laplace transform provides a traveling wave solution to the KPP equation
$$\frac{\partial u}{\partial t}=\frac{1}{2}\frac{\partial^2u}{\partial x^2}+\beta(f(u)-u),$$
where $f(u):=\sum^\infty_{n=0}p_n u^n$ is the generating function of the distribution $\{p_n, n\geq 0\}$.
By a traveling wave solution we mean a solution of the form
$u(t,x) = w(x-ct)$, where $w$ is a monotone function connecting $0$ at $-\infty$ to $1$ at $+\infty$ and $c$ is called the speed of the wave.
The following result of Yang and Ren \cite{YR}  gives a necessary and sufficient condition for $\partial W(\underline{\lambda})$ being non-degenerate.

\begin{thm}\label{deriv-m} Suppose $p_0=0$ and $\lambda=\underline{\lambda}$. For any $x\in\mathbb{R}$, we have

(1)  $\partial W(\underline{\lambda})>0$ $\mathbf P_x$-almost surely  if $\sum^\infty_{n=1}p_nn(\log n)^2<\infty$;

(2) $\partial W(\underline{\lambda})=0$ $\mathbf P_x$-almost surely if  $\sum^\infty_{n=1}p_nn(\log n)^2=\infty$.
\end{thm}

\begin{corollary}\label{wave} When $c=\underline{\lambda}=\sqrt{2\beta\beta(A-1)}$ and
$\sum^\infty_{n=1}p_nn(\log n)^2<\infty$,
there is a unique traveling wave with speed $c$ given by
$\Phi_c(x)=E^x\exp(-e^{\lambda x}\partial W(\lambda)).$
\end{corollary}

Now  using our spine decomposition in Section \ref{Spinedecom}, the above Theorem \ref{deriv-m} and Corollary \ref{wave} also hold  for the case that $p_0>0$.

\subsection{Typed branching Brownian motion and traveling wave solution}
Hardy and Harris \cite{HH1} considered  a typed branching diffusion in which the Hunt process,
that is, the spatial motion,
is described by $(Y_t, \eta_t)_{t\geq 0}$, where the type $\eta_t$ evolves as a Markov chain on $I:=\{1,\cdots, n\}$ with
$Q$-matrix $\theta Q$, where $\theta>0$ is a constant, and
the spatial location, $S_t$, moves as a driftless Brownian motion on $\mathbb{R}$ with diffusion coefficient $a(i)>0$ whenever $\eta_t$ is in state $i$. Any
particle currently of type $i$ will undergo fission at rate $\beta(i)$ to be replaced by a random number of offspring with law $\{p_n(i), n\geq 0\}$.
At birth, offspring inherit the parent's spatial and type positions and then move off independently, repeating stochastically the parent's behaviour, and so on. Let $A(i):=\sum^\infty_{n=0}np_n(i)<\infty$ be the mean of the distribution of offspring  given by a type $i$ particle.

As usual, let the configuration of the whole branching diffusion at time
$t$ be given by the $\mathbb{R}\times I$-valued point process $X_t =\sum_{u\in L_t}\delta_{(y,i)}$, where $L_t$ is the set of particles alive at time $t$. Let the probabilities for this process be given by $\{\mathbf P_{(y,i)}, (y,i)\in \mathbb{R}\times I\}$ defined on the natural filtration,
$({\cal F}_t, t\geq 0\}$, where $\mathbf P_{(y,i)}$ is the law of the typed branching Brownian motion starting with one
initial particle of type $i$ at spatial position $y$.

For this finite-type branching diffusion, a fundamental positive martingale is defined for this model:
$$
W_\lambda(t):=\sum_{u\in L_t}v_\lambda(\eta_u(t))e^{\lambda Y_u(t)-
E_\lambda t},\quad t\ge 0,
$$
where $v_\lambda$ and $E_\lambda$ satisfy
$$
\left(\frac{1}{2}\lambda^2\Sigma+\theta Q+(A-1)R\right)v_\lambda=E_\lambda v_\lambda,
$$
where $\Sigma:=\mbox{diag}(a(i): i\in I)$, $A:=\mbox{diag}(A(i), i\in I)$ and $R:=\mbox{diag}(\beta(i): i\in I)$. That is, $v_\lambda$ is
the (Perron-Frobenius) eigenvector of the matrix $\frac{1}{2}\lambda^2A+\theta Q+R$, with eigenvalue $E_\lambda$. This martingale should be compared with the corresponding martingale \eqref{martingale} for  branching Brownian motion.

Since $\{W_t(\lambda), t\ge 0\}$ is a strictly-positive martingale
it is immediate that $W_\infty(\lambda):=\lim_{t\to\infty}W_t(\lambda)$ exists and is finite almost-surely under $\mathbf P_{(y,i)}$. Under the condition that $p_0(i)=0, i\in I$,  Hardy and Harris \cite[Theorem 10.4]{HH1} give a necessary and sufficient conditions for
$L^1$-convergence of the martingale $\{W_t(\lambda), t\ge 0\}$.
Using our general spine decomposition, the condition that  $p_0(i)=0, i\in I$ can be dropped now. Once again, in many cases where the martingale has a non-trivial limit,
the convergence will be much stronger than merely in $L^1(\mathbf P_{(y,i})$, that is $L^p(\mathbf P_{(y,i}) (p\in(1,2])$-convergence, using the spine decomposition, see  \cite[Theorem 10.5]{HH1}.

In Harris and Williams \cite{HW}, a continuous-typed branching diffusion, where the Hunt process,
that is, the movement of the particles,
is described by $(Y_t, V_t)_{t\geq 0}$,
 where the spatial motion $(Y_t)_{t\geq 0}$ is a driftless Brownian motion
with instantaneous variance $ay^2$ with $a\geq 0$ being a fixed constant, and the type  moves on the real line as an Orstein-Uhlenbeck process.   A particle of type $v$ dies at rate $rv^2+\rho$ with $r,\rho\geq 0$ being fixed constant, and then produce two particles at the same space-type location as the parent.  This model is similar in flavour
to the  finite-type model. There is also a strictly-positive martingale
$\{W_t(\lambda), t\ge 0\}$.
Hardy and Harris \cite[Theorem 11.1]{HH1}, using the spine technique,
give a necessary condition and sufficient condition for
$L^p$-convergence with $p\in (1,2]$.
 We remark here that, their results also hold for general offspring distribution, that is when a particle of type $v\in\mathbb{R}$ dies, it gives birth a random number of particles according to law $\{p_n(v), n\geq 0\}$ at the same space-type location as the parent.  Under some moment condition on $\{p_n(v),n\geq 0\}, v\in\mathbb{R}$,  allowing $p_0(v)=0, v\in \mathbb{R}$, results similar to \cite[Theorem 11.1]{HH1} remain true. We will not go to the details here.

\bigskip

\noindent {\bf Acknowledgment:} We thank the referees for helpful comments.

\bigskip

\begin{singlespace}

\end{singlespace}
\end{doublespace}
\vskip 0.3truein \vskip 0.3truein

\bigskip
 \noindent {\bf Yan-Xia Ren:} LMAM School of Mathematical Sciences, Peking
 University,  Beijing, 100871, P. R. China,
  E-mail: {\tt yxren@math.pku.edu.cn} \\

\bigskip
\noindent {\bf Renming Song:} Department of Mathematics,
 The University of Illinois,  Urbana, IL 61801 U.S.A.,
  E-mail: {\tt rsong@illinois.edu} \\
 \bigskip

\end{document}